\documentclass[11pt, a4paper]{amsart}
\usepackage{times}
\usepackage{a4wide}
\usepackage[hidelinks]{hyperref}
\usepackage[british]{babel}
\usepackage{enumerate, longtable, multirow}
\usepackage{amsmath, amscd, amsfonts, amsthm, amssymb,color}
\usepackage[all]{xypic}
\usepackage[T1]{fontenc}

\newtheorem{thm}{Theorem}[section]
\newtheorem{lem}[thm]{Lemma}
\newtheorem{defi}[thm]{Definition}
\newtheorem{rem}[thm]{Remark}

\newtheorem{prop}[thm]{Proposition}
\newtheorem{cor}[thm]{Corollary}
\newtheorem*{cor-non}{Corollary}

\newcommand{\End}{\mathrm{End}}
\newcommand{\Hom}{\mathrm{Hom}}
\newcommand{\Fr}{\mathrm{Fr}}
\newcommand{\GL}{\mathrm{GL}}

\newcommand{\minim}{\mathrm{min}}

\newcommand{\Katz}{\mathrm{Katz}}
\newcommand{\et}{\mathrm{et}}
\newcommand{\Ver}{\mathrm{Ver}}

\DeclareMathOperator{\Norm}{N}
\DeclareMathOperator{\Spec}{Spec}

\DeclareMathOperator{\Image}{im}
\DeclareMathOperator{\Gal}{Gal}

\DeclareMathOperator{\Frob}{Frob}
\DeclareMathOperator{\KS}{KS}
\DeclareMathOperator{\tr}{tr}

\newcommand{\ilim}[1]{\underset{#1}{\underrightarrow{\lim} \;}}
\newcommand{\diam}[1]{{\langle #1 \rangle}}

\newcommand{\cA}{\mathcal{A}}
\newcommand{\cB}{\mathcal{B}}
\newcommand{\cC}{\mathcal{C}}

\newcommand{\cF}{\mathcal{F}}

\newcommand{\cI}{\mathcal{I}}

\newcommand{\cO}{\mathcal{O}}

\newcommand{\cX}{\mathcal{X}}

\newcommand{\rG}{\mathrm{G}}
\newcommand{\rH}{\mathrm{H}}
\newcommand{\rI}{\mathrm{I}}

\newcommand{\rN}{\mathrm{N}}

\newcommand{\fa}{\mathfrak{a}}
\newcommand{\fb}{\mathfrak{b}}
\newcommand{\fc}{\mathfrak{c}}
\newcommand{\fd}{\mathfrak{d}}

\newcommand{\fm}{\mathfrak{m}}
\newcommand{\fn}{\mathfrak{n}}
\newcommand{\fo}{\mathfrak{o}}
\newcommand{\fp}{\mathfrak{p}}
\newcommand{\fq}{\mathfrak{q}}
\newcommand{\fr}{\mathfrak{r}}

\newcommand{\CC}{\mathbb{C}}

\newcommand{\FF}{\mathbb{F}}
\newcommand{\GG}{\mathbb{G}}

\newcommand{\NN}{\mathbb{N}}

\newcommand{\QQ}{\mathbb{Q}}

\newcommand{\TT}{\mathbb{T}}

\newcommand{\ZZ}{\mathbb{Z}}

\newcommand{\Qbar}{\overline{\QQ}}

\newcommand{\Fbar}{\overline{\FF}}

\newcommand{\sigmabar}{{\overline{\sigma}}}
\newcommand{\Sigmabar}{{\overline{\Sigma}}}
\newcommand{\rhobar}{{\overline{\rho}}}

\newcommand{\tautilde}{{\widetilde{\tau}}}

\newcommand{\ps}{{\mathrm{ps}}}
\newcommand{\ord}{{\mathrm{ord}}}

\newcommand{\cusp}{{\mathrm{cusp}}}

\renewcommand{\ord}{{\mathrm{ord}}}
\newcommand{\kum}{{\mathrm{Kum}}}
\newcommand{\DP}{{\mathrm{DP}}}
\newcommand{\Ra}{{\mathrm{Ra}}}
\newcommand{\AG}{{\mathrm{AG}}}
\newcommand{\RX}{{\mathrm{RX}}}

\date{\today}

\begin{document}

\begin{abstract}
We prove that the Galois pseudo-representation valued in the mod $p^n$ cuspidal Hecke algebra for $\GL(2)$ over a totally real number field $F$, of parallel weight $1$  and level prime to $p$,  is unramified at any place above~$p$. The same is true for the non-cuspidal Hecke algebra at  places above~$p$ whose ramification index is not 
divisible by  $p-1$.  A novel geometric ingredient, which is also of independent interest, is the construction and study, in the case when $p$ ramifies in $F$, of generalised $\Theta$-operators using Reduzzi--Xiao's generalised Hasse invariants, including especially an injectivity criterion in terms of minimal weights.
\end{abstract}


\title{Unramifiedness of weight $1$ Hilbert Hecke algebras}
\author{Shaunak V. Deo}
\address{Department of Mathematics, Indian Institute of Science, Bangalore 560012, India}
\email{shaunakdeo@iisc.ac.in}
\author{Mladen Dimitrov}
\address{University of Lille, CNRS, UMR 8524 -- Laboratoire Paul Painlev\'e, 59000 Lille, France.}
\email{mladen.dimitrov@gmail.com }
\author{Gabor Wiese}
\address{University of Luxembourg, Department of Mathematics, Maison du nombre, 6~avenue de la Fonte, L-4364 Esch-sur-Alzette, Luxembourg }
\email{gabor.wiese@uni.lu}
\maketitle

\footnotesize
\tableofcontents
\normalsize

\section*{Introduction}

Starting with Wiles \cite{Wiles} and Taylor--Wiles \cite{TaylorWiles}, $R=\TT$ theorems have been developed and taken a role as cornerstones in number theory. They provide both the existence of Galois representations with values in Hecke algebras satisfying prescribed local properties and modularity lifting theorems.
The state of $R=\TT$ theorems for $2$-dimensional  representations in residual characteristic~$p$ of the absolute Galois group $\rG_\QQ$ of $\QQ$ and Hecke algebras acting on elliptic modular forms is quite satisfactory. In particular, the notoriously difficult case of Galois representations that are unramified at an odd prime~$p$ 
has been settled by ground-breaking work of Calegari and Geraghty \cite{CaGe}, in which they show that  those correspond to modular forms of weight $1$. 
More precisely, given an odd irreducible representation $\bar\rho: \rG_\QQ \to \GL_2(\Fbar_p)$ unramified outside 
a finite set  of places $S$ not containing $p$, they show that 
\[R^S_{\QQ,\bar\rho}\xrightarrow{\sim} \TT^{(1)}_{\bar\rho},\]
where $R^S_{\QQ,\bar\rho}$ is the universal deformation ring parametrising deformations of $\bar\rho$ 
which are unramified outside $S$ and $\TT^{(1)}_{\bar\rho}$ is 
the local component at $\bar\rho$ of a weight $1$ Hecke algebra of a certain level prime to~$p$.

In this article, we address the corresponding question for parallel weight $1$ Hilbert modular forms over  a totally real field~$F$ of degree $d= [F:\QQ] \geqslant 2$ and  ring of integers $\fo$. 
We focus on the construction of the Galois (pseudo-)representation with values in the parallel weight $1$ Hecke algebra with $p$-power torsion coefficients and proving its local ramification properties.
In particular, given a finite set $S$ of places in $F$ relatively prime to~$p$ and a totally odd irreducible representation $\bar\rho: \rG_F \to \GL_2(\Fbar_p)$ unramified outside $S$ we show that there exists a surjective homomorphism 
\[R^S_{F,\bar\rho}\twoheadrightarrow \TT^{(1)}_{\bar\rho},\]
where $R^S_{F,\bar\rho}$ is the universal deformation ring parametrising deformations of $\bar\rho$ 
which are unramified outside $S$  and $\TT^{(1)}_{\bar\rho}$ is 
the local component at $\bar\rho$ of a weight $1$ Hecke algebra of a certain level prime to $p$ (see Corollary~\ref{cor:RT} for a precise statement). 

Let  $M_\kappa(\fn,R)$ be the $R$-module of   Hilbert modular forms of 
parallel weight $\kappa\geqslant  1$ and prime to $p$ level $\fn$ over a $\ZZ_p$-algebra $R$, as in Definition~\ref{defi:HMFWiles}.
This $R$-module is equipped with a commuting family of Hecke operators $T_\fq$ as well as with diamond operators $\langle \fq \rangle$ for all  primes $\fq$ of~$F$ not dividing~$\fn$.
Let $K/\QQ_p$ be a finite extension containing the images of all embeddings of $F$ in $\Qbar_p$, and let  $\cO$ be its 
valuation ring, $\varpi$ a uniformiser and $\FF = \cO/\varpi$ its residue field.
We put $M_\kappa(\fn ,K/\cO) = \ilim{n} M_\kappa(\fn ,\cO/\varpi^n)$ and  define the parallel weight $1$ Hecke algebra 
\[ \TT^{(1)}  =\Image\big(\cO[T_\fq, \diam{\fq}]_{\fq \nmid \fn p} \to \End_\cO(M_1(\fn ,K/\cO))\big),\]
as well as its cuspidal quotient $\TT^{(1)}_\cusp$ acting faithfully on the submodule of parallel weight $1$  cuspforms. 
We can now state the main results of this article. Let $p\fo = \prod_{\fp \mid p} \fp^{e_\fp}$ 
with $e_\fp\geqslant 1$. We emphasise that there is no restriction on the ramification of~$p$ in~$F$.

\begin{thm}\label{thm:main}
There exists  a $\TT^{(1)}$-valued pseudo-representation $P^{(1)}$ of $\rG_F$ of degree $2$ which is unramified at all primes $\fq$ not dividing $\fn p$ and satisfies $P^{(1)}(\Frob_\fq)=(T_\fq, \diam{\fq})$.

Moreover, if $p-1$ does not divide $e_\fp$ for some $\fp \mid p$, then $P^{(1)}$ is also unramified at $\fp$ and satisfies
$P^{(1)}(\Frob_\fp)=(T_\fp, \diam{\fp})$, in particular $T_{\fp}\in \TT^{(1)}$.  

Finally, the pseudo-representation $P_\cusp^{(1)}$ obtained after composing $P^{(1)}$ with the natural surjection 
$\TT^{(1)}\to \TT_\cusp^{(1)}$ is unramified at all  $\fp \mid p$  and satisfies $P_\cusp^{(1)}(\Frob_\fp)=(T_\fp, \diam{\fp})$. 
\end{thm}

The strategy of the proof is  based on the {\em doubling} method developed in \cite{Wi},  further simplified and conceptualised in \cite{DiWi} and  \cite{CaGe}. The parallel weight $1$ Hilbert modular forms over $\cO/\varpi^n$ can be mapped into some higher weight in two ways, per prime $\fp$ dividing $p$, either by multiplication by a suitable power of the total Hasse invariant, or  by applying a $V$-operator. 
That  \emph{doubling} map is used by Calegari--Specter in  \cite{CaSp} to prove an analogue of Theorem~\ref{thm:main} 
when $F=\QQ$, for which they successfully develop the notion of a $p$-{\em ordinary}  pseudo-representation. 
In that case, one knows by a result of Katz that the doubling map is  injective. 
Furthermore,  the existence of the  Hecke operator $T_p$ acting on weight $1$ modular forms and the knowledge of its precise effect on the $q$-expansion (both due to Gross) allow one to show that the image of the doubling map is contained in the $p$-ordinary part of the higher weight space.

The existence of an optimally integral Hecke operator $T_\fp$ acting on  parallel weight $1$ Hilbert modular forms with arbitrary coefficients having the desired effect on their $q$-expansions (see \cite{DiKS} improving on and correcting previous works such as \cite{ERX} and \cite{DiWi}) allows us to adapt  the overall   Calegari--Specter strategy to the Hilbert modular setting, while slightly generalising and clarifying some aspects of their arguments (see  \S\ref{s:pseudo-CS}), the main challenge being to prove the injectivity of the doubling map. Note that the 
simple calculation  in \cite{DiWi} showing  that injectivity after restriction to an eigenspace is  insufficient as the Hecke algebra modulo $p$ need not be semi-simple. 
Instead, we observe that the  injectivity of the doubling map would follow from the injectivity of a certain generalised $\Theta$-operator, introduced in the  foundational work \cite{AG} of Andreatta--Goren for Hilbert modular forms in characteristic $p$ defined over the Deligne--Pappas moduli space. When $p$ is unramified in $F$, the theory of partial $\Theta$-operators was also developed by Diamond--Sasaki  \cite{DS} in a more general setting using a slightly different approach from that of \cite{AG}.
However, when $p$ is ramified in $F$, the results of \cite{DS} do not apply, while those of \cite{AG} are not sufficiently precise for our purposes, as the Hilbert modular forms defined over the Deligne--Pappas model `miss' some weights, and as a consequence the injectivity result of the latter paper is not optimal. 
In order to tackle this problem, we go back to the root of the problem and work with the Pappas--Rapoport moduli space, which does not miss any weight.

Capitalising on the theory of generalised Hasse invariants developed by Reduzzi--Xiao \cite{RX}
in this context, we carefully  revisit  \cite{AG} and  develop in \S\ref{sec:theta} the needed theory of generalised  $\Theta$-operators  over the Pappas--Rapoport moduli space and  prove a refined injectivity criterion in terms of the minimal weights. In particular, we show that the generalised $\Theta$-operators are indeed injective on parallel weight $1$  Hilbert modular forms provided their weight is minimal at~$\fp$.
By the recent works \cite{DK,DK2} of Diamond and Kassaei (see \S\ref{sec:min}) weight $1$ Hilbert modular forms having `non-minimal'
weight at~$\fp$ could only possibly exist when $p-1$  divides $e_\fp$, and are products of forms of partial weight $0$  at~$\fp$ with generalised  Hasse invariants. 

In order to show the vanishing of the space of Katz cuspforms of partial weight $0$ at~$\fp$, and thus complete the proof of 
the second and third parts of the Theorem, in \S\ref{sec:partial-frob} we construct a  partial Frobenius endomorphism $\Phi_{\fp^e}$ of this space and show that it is simultaneously injective and nilpotent.
Our construction is inspired by the one in  Diamond--Sasaki \cite[\S9.8]{DS} in the case when $p$ is unramified in $F$.  
We also compute its effect on $q$-expansions, which is crucially used in our proof and,  in order to avoid having to switch between different cusps, we only study the partial Frobenius operator of an  appropriate power of $\fp$, rather than of $\fp$ itself. 

In the language of linear representations, we prove the following result, which can be seen as a first step towards an $R=\TT$ theorem.

\begin{cor-non}[Corollary~\ref{cor:RT}]
For every non-Eisenstein maximal ideal  $\fm$ of $\TT^{(1)}$ (see Definition \ref{def:eis}) there exists a   representation
\[\rho_\fm: \rG_F \to \GL_2(\TT^{(1)}_{\fm}),\]
unramified at all primes $\fq$ not dividing $\fn$ such that  $\tr(\rho_\fm(\Frob_\fq)) =T_\fq$ and   $\det(\rho_\fm(\Frob_\fq)) = \diam{\fq}$. 
\end{cor-non}

We believe that our modest  contribution to the theory of generalised $\Theta$-operators in the setting of the Pappas--Rapoport splitting model
is worthwhile on its own, beyond the application to our main theorem.
On our way to the injectivity criterion, we also explore some related themes, such as the relation between Hilbert modular forms defined over the Pappas--Rapoport model with those defined over the Deligne--Pappas model, and the $q$-expansion and vanishing loci of the generalised Hasse invariants defined by Reduzzi--Xiao. 
We hope that it bridges the gap between many existing references in the literature and also clarifies some important aspects of the theory of mod $p$ Hilbert modular forms.
In the meantime, motivated by geometric Serre weight conjectures, 
Diamond \cite{D} extended the techniques of \cite{DS} to also construct partial $\Theta$-operators which have an optimal effect on weights in the case where $p$  ramifies in $F$.  Moreover, Diamond generalised in \cite{D} the construction of the partial Frobenius operators (our partial Frobenius operator $\Phi_{\fp^e}$ is essentially Diamond's  $V_{\fp}^e$). Note that Diamond also describes kernels of partial $\Theta$-operators in terms of images of his partial Frobenius maps $V_{\fp}$ (see \cite[Thm. 9.1.1]{D}).
However, our construction is less technical because we restrict to the Rapoport locus and we only consider the case of weights $0$ at~$\fp$. 

\subsection*{Acknowledgements}
{\small
The authors are indebted to Fabrizio Andreatta, Adel Betina, Frank Calegari, Fred Diamond, Payman Kassaei, Sheng-Chi Shih and Liang Xiao for many clarifying explanations and discussions.
The debt to the works of Andreatta--Goren \cite{AG} and Reduzzi--Xiao~\cite{RX} is evident.

The authors would particularly like to thank the anonymous referee for the careful reading pointing out several inaccuracies and for the insightful comments which helped tremendously improve the manuscript.

The research leading to this article is jointly funded by the Agence National de Recherche ANR-18-CE40-0029 and the Fonds National de Recherche Luxembourg INTER/ANR/18/12589973 in the project {\it Galois representations, automorphic forms and their L-functions} (GALF).

The first two authors would like to thank IISER Pune and ICTS Bangalore, where part of the work was done, for their hospitality, namely during the  program {\it  Perfectoid spaces } (Code: ICTS/perfectoid2019/09).
The first author was partially supported by a Young Investigator Award from the Infosys Foundation, Bangalore and also by the DST FIST program - 2021 [TPN - 700661]. 
 }

\subsection*{Notation}
Throughout the paper, we will use the  following  notation.
We let $F$ be a totally real number field of degree $d \geqslant  2$ and  ring of integers $\fo$. 
We denote by $\Qbar\subset \CC$ the subfield of algebraic numbers and denote by 
$\rG_F = \Gal(\Qbar /F)$ the absolute Galois group of $F$. For every prime $\fq$ of $F$ we denote by $\Frob_\fq \in \rG_F$ a fixed choice of an arithmetic Frobenius at~$\fq$. Let $\fp$ be a prime of $F$ dividing $p$. 
Fixing an  embedding $\iota_p$  of $\Qbar$ into a fixed algebraic closure $\Qbar_p$ of $\QQ_p$ 
allows one to see the absolute Galois group $\rG_{F_\fp}=\Gal(\Qbar_p/F_\fp)$  of $F_\fp$
as a decomposition subgroup of $\rG_F$ at $\fp$, and we let $\rI_\fp$ denote its inertia subgroup. 
Furthermore, we fix a  finite extension $K/\QQ_p$ containing the images of all embeddings of $F$ in $\Qbar_p$, and let  $\cO$ be its 
valuation ring, $\varpi$ a uniformiser and $\FF = \cO/(\varpi)$ its residue field.

For a prime $\fp$ of $F$ dividing $p$, denote the residue field of $F_{\fp}$ by $\FF_{\fp}$ and the ring of Witt vectors of $\FF_{\fp}$ by $W(\FF_{\fp})$.  We also let $f_\fp$ and $e_\fp$ denote the residue and the ramification index of $\fp$, respectively.
Let $\Sigma$ be the set of infinite places of~$F$, which we view as embeddings $F \hookrightarrow \Qbar_p$ via $\iota_p$.
We have a natural partitioning $\Sigma = \coprod_{\fp \mid p} \Sigma_\fp$ where $\Sigma_\fp$ contains exactly those embeddings
inducing the place~$\fp$.
For $\sigma \in \Sigma_\fp$, we denote by $\sigmabar$ its restriction to  the maximal unramified subfield of $F_\fp$ or, equivalently, the induced embedding of $\FF_{\fp(\tau)}$ into $\Fbar_p$.
Furthermore, we let $\Sigmabar_\fp = \{ \sigmabar \;|\; \sigma \in \Sigma_\fp\}$ and $\Sigmabar = \{ \sigmabar \;|\; \sigma \in \Sigma\} = \coprod_{\fp \mid p} \Sigmabar_\fp$.
As a general rule, elements of $\Sigma$ will be called~$\sigma$ whereas $\tau$ usually designates an element of~$\Sigmabar$.
In both cases, $\fp(\sigma)$ and $\fp(\tau)$ denotes the underlying prime ideal.
When either $\sigma$ or $\tau$ is clear from the context, we will just denote this prime ideal by $\fp$. 
In particular, an element $\tau \in \Sigmabar_\fp$ denotes both an embedding $\FF_{\fp(\tau)} \hookrightarrow \FF$ and the corresponding $p$-adic one
$W(\FF_{\fp(\tau)}) \hookrightarrow \cO$.
Denoting the absolute arithmetic Frobenius on $\FF$ by $\phi$, we have $\Sigmabar_\fp = \{ \phi^j  \circ \tau \;|\; j \in \ZZ\} \simeq
\ZZ/f_\fp \ZZ \simeq \Gal(\FF_\fp/\FF_p)$ for any choice $\tau \in \Sigmabar_\fp$.
For any $\tau \in \Sigmabar_\fp$, we let $\Sigma_\tau = \{ \sigma \in \Sigma_\fp \;|\; \sigmabar = \tau\} = \{\sigma_{\tau,i} \;|\; 1 \leqslant i \leqslant e_\fp\}$, where
the numbering is chosen in an arbitrary, but fixed way.
As an abbreviation, we write $\tautilde  = \sigma_{\tau,e_\fp}$.

Let $\cC$ be a fixed set of representatives, all relatively prime to $p$,  for the   narrow class group of~$F$.

\section{Hilbert modular forms in finite characteristic}\label{section1}

This section refines the theory of $\Theta$-operators developed by Andreatta--Goren in \cite{AG}, when $p$ ramifies in~$F$, in the setting of Hilbert modular forms defined over the Pappas--Rapoport splitting models for Hilbert modular varieties with the aim of proving the injectivity of the doubling map in~\S\ref{sec:doubling}.
Along the way, we will need the generalised Hasse invariants of Reduzzi and Xiao \cite{RX}, results of Diamond and Kassaei \cite{DK, DK2} about minimal weights as well as a partial Frobenius operator generalised from~\cite{DS}.

Throughout this  section we  fix   an ideal  $\fn$ of~${\fo}$  relatively prime to~$p$ and  having a prime factor which does not divide $6\fd$, where      $\fd$ denotes the different of $F$.

\subsection{Pappas--Rapoport splitting models for Hilbert modular varieties}
\label{prsection}

Since we allow our base field $F$ to ramify at~$p$, we have to be careful with the model we choose for our Hilbert modular variety. 

Fix $\fc \in \cC$. We first consider the functor from the category of locally Noetherian $\ZZ_p$-schemes to the category of sets which assigns to a scheme $S$ the set of isomorphism classes of tuples $(A,\lambda,\mu)$ where:
\begin{enumerate}[(i)]
\item $A$ is  a {\em Hilbert-Blumenthal Abelian Variety (HBAV)} over $S$, {\it i.e.},  an abelian $S$-scheme of relative dimension $d$, together with a ring embedding ${\fo} \hookrightarrow \End_S(A)$.
\item $\lambda$ is a $\fc$-polarisation of $A/S$, {\it i.e.}, an isomorphism $\lambda : A^{\vee} \to A \otimes_{{\fo}} \fc$ of HBAV's over $S$ such that the induced isomorphism $\Hom_{{\fo}}(A, A \otimes_{{\fo}}\fc) \simeq \Hom_{{\fo}}(A,A^{\vee})$ sends elements of $\fc$ (resp.  
 of the cone  $\fc_+$   of its totally positive elements) 
to symmetric elements (resp. to polarisations),
\item $\mu$ is a $\mu_{\fn}$-level structure on $A$, {\it i.e.}, an ${\fo}$-linear closed embedding of $S$-schemes $\mu : \mu_{\fn} \to A$,
where $\mu_{\fn}$ denotes the Cartier dual of the constant group scheme ${\fo}/\fn$ over $S$.
\end{enumerate}
Under our assumption on $\fn$ above, this functor is representable  by a  $\ZZ_p$-scheme $\cX^{\DP}$ of finite type, 
called the  {\em Deligne--Pappas moduli space} (see \cite[Rem.~3.3]{AG} and \cite[Lem.~1.4]{DiTi}).

Suppose now that $A$ is an HBAV over a locally Noetherian $\cO$-scheme $S$ with structure map $s : A \to S$ 
and let $\Omega^1_{A/S}$ be the sheaf of relative differentials of $A$ over $S$. Define
\[\omega_S = s_*\Omega^1_{A/S},\]
{\it i.e.}, $\omega_S$ is the sheaf of invariant differentials of $A$ over~$S$.
Consider the decomposition
\begin{align}\label{eq:dec-O}
{\fo} \otimes_{\ZZ} \cO_S
= ({\fo} \otimes_{\ZZ} \ZZ_p) \otimes_{\ZZ_p} \cO_S =\prod_{\fp \mid p} \fo_{\fp} \otimes_{\ZZ_p} \cO_S
=\prod_{\tau \in \Sigmabar}\fo_{\fp(\tau)} \otimes_{W(\FF_{\fp(\tau)}),\tau} \cO_S.
\end{align}
 It implies that  we have a corresponding decomposition
\begin{align}\label{eq:dec-omega}
\omega_S = \bigoplus_{\tau \in \Sigmabar} \omega_{S,\tau}.
\end{align}
The sheaf $\omega_{S,\tau}$ is locally free over $S$ of rank $e_{\fp(\tau)}$ (see \cite[\S2.2]{RX}).
Note that on $\omega_{S,\tau}$, the action of $W(\FF_{\fp(\tau)}) \subset \fo_{\fp(\tau)}$ is via~$\tau$.
Fix a uniformiser $\varpi_{\fp(\tau)}$ of $\fo_{\fp(\tau)}$.
From the product decomposition above, we get an action of $\fo_{\fp(\tau)}$ on $\omega_{S,\tau}$.
Denote the action of $\varpi_{\fp(\tau)}$ on $\omega_{S,\tau}$ by $[\varpi_{\fp(\tau)}]$.
\smallskip

We are now ready to present the  Pappas--Rapoport model. 
Consider the functor from the category of locally Noetherian $\cO$-schemes to the category of sets which assigns to a scheme $S$ the set of isomorphism classes of tuples $(A,\lambda,\mu,(\cF_\fp)_{\fp \mid p})$ where $(A,\lambda,\mu)$ is as above and for all $\fp \mid p$, $\cF_\fp$ is a collection $(\cF_\tau^i)_{\tau \in \Sigmabar_\fp, 0 \leqslant  i \leqslant  e_\fp}$ of $\fo\otimes \cO_S$-modules, which are locally free as $\cO_S$-modules, such that:
\begin{itemize}
\item $0 = \cF^0_\tau \subset \cdots \subset \cF^{e_\fp}_\tau = \omega_{S,\tau}$,
\item for any $\sigma =\sigma_{\tau,i} \in \Sigma_\tau$, the $\cO_S$-module 
$\omega_{S,\tau,i}=\omega_{S,\sigma}=\cF^i_\tau/\cF^{i-1}_\tau$ is locally free of rank $1$ and annihilated by $[\varpi_\fp] - \sigma(\varpi_\fp)$.
Note that the numbering here depends on the one for $\Sigma_\tau$.
\end{itemize} 
This functor is representable by a smooth  $\cO$-scheme $\cX$ of finite type called   the {\em Pappas--Rapoport moduli space} (see \cite[Prop.~2.4]{RX} and \cite[Lem.~1.4]{DiTi}).

In order to better understand the relation between the Deligne--Pappas and the Pappas--Rapoport moduli spaces, 
we recall that the {\em Rapoport locus} $\cX^\Ra$  is the open subscheme of $\cX^{\DP}$ classifying HBAV's  $s : A \to S$ satisfying the following condition introduced by Rapoport: $s_*\Omega^1_{A/S}$ is a locally free ${\fo} \otimes_{\ZZ} \cO_S$-module of rank~$1$.
Then $\cX^\Ra$ is the smooth locus of~$\cX^{\DP}$ and its complement  is supported in the special fibre and has codimension at least $2$ in it. The forgetful map $\cX \to \cX_{\cO}^{\DP}$ induces an isomorphism on the open subscheme $\cX_{\cO}^\Ra$ (see \cite[Prop.~2.4]{RX}). If $p$ is unramified in $F$, the different schemes agree: $\cX=\cX_{\cO}^\Ra=\cX_{\cO}^{\DP}$ (see \cite[\S1]{RX}).

Let $\cA$ be the universal abelian scheme over $\cX$ with structure morphism $s : \cA \to \cX$.
Let $\omega_{\cX} = s_*\Omega^1_{\cA/\cX}$.
Note that the restriction of $\omega_{\cX}$ to $\cX_{\cO}^\Ra$ is a locally free sheaf of rank $1$ over ${\fo} \otimes_{\ZZ} \cO_{\cX^\Ra_{\cO}}$.
As abbreviation we write $\omega, \omega_\tau, \omega_{\tau,i}, \omega_\sigma$ for $\omega_{\cX}, \omega_{\cX,\tau}, \omega_{\cX,\tau,i}, \omega_{\cX,\sigma}$.
In particular, for each $\tau \in \Sigmabar$, the sheaf $\omega_\tau$ is equipped with a filtration the graded pieces of which are the invertible sheaves $\omega_\sigma$ for $\sigma \in \Sigma_\tau$. In \cite{RX} this is referred to as the {\em universal filtration}.
We point out explicitly that the last graded piece $\omega_\tautilde$ is a quotient of $\omega_\tau$.

Next we give, following Katz, a geometric definition of the space of Hilbert modular forms.

\begin{defi}\label{defi:HMF}
A {\em Katz Hilbert modular form} of weight $k = \sum_{\sigma \in \Sigma} k_{\sigma}\sigma \in \ZZ[\Sigma]$, level~$\fn$
and coefficients in an $\cO$-algebra $R$ is a global section of the line bundle
$\omega^{\otimes k}=\bigotimes_{\sigma \in \Sigma}\omega^{\otimes k_{\sigma}}_{\sigma}$ over $\cX\times_\cO R$.
We will denote  by $M_k^\Katz(\fc,\fn;R)$ the corresponding $R$-module.

Its $R$-submodule of cuspforms $S_k^\Katz(\fc,\fn;R)$ consists of those Katz Hilbert modular forms that vanish along the cuspidal divisor of any toroidal compactification of $\cX\times_\cO R$ (see \cite[\S2.11]{RX}). 
\end{defi}

As $\cX$ admits toroidal compactifications (see \cite[\S2.11]{RX}) which are smooth and proper over~$\cO$ and to which $\omega_\sigma$ extends for all~$\sigma \in \Sigma$,  the Koecher principle implies, in view of \cite[\href{https://stacks.math.columbia.edu/tag/02O5}{Tag 02O5}]{stacks-project},  that $M_k^\Katz(\fc,\fn;R)$ is a finitely generated $R$-module.

\begin{rem}\label{rem:HMF-parallel}
When the weight $k\in \ZZ[\Sigma]$ is \emph{parallel}, {\it i.e.},  $k_\sigma = \kappa\in \ZZ$ for all $\sigma \in \Sigma$, 
one also could define a {\em Katz Hilbert modular form} of parallel weight  $\kappa\in \ZZ$, level~$\fn$ and coefficients in a $\ZZ_p$-algebra $R$
as a global section of the line bundle $\left(\bigwedge^d s_* \Omega^1_{\cA/\cX^{\DP}}\right)^{\otimes\kappa}$ 
over $\cX^{\DP}\times_{\ZZ_p} R$. By Zariski's Main Theorem applied to the proper birational map $\cX \to \cX_{\cO}^{\DP}$ between normal varieties, this would lead to the same space as in  Definition~\ref{defi:HMF}.
\end{rem}

\subsection{Generalised Hasse invariants}\label{subsec:Hasse}

From this point onwards we will work over~$\FF$.
Let $X$ be the Pappas--Rapoport moduli space over~$\FF$, {\it i.e.}, the special fibre $\cX\times_\cO\FF$ of~$\cX$.
There is a natural morphism $X \to \cX^{\DP}\times_{\ZZ_p}\FF$ obtained by forgetting the filtrations.
Let $X^\Ra=\cX^\Ra\times_{\ZZ_p}\FF$. We have the equality
\begin{align}\label{eq:prod}
{\fo} \otimes_{\ZZ} \FF =\prod_{\fp\mid p}\fo_{\fp} \otimes_{\ZZ_{p}} \FF \simeq\prod_{\tau\in \Sigmabar} \fo_{\fp(\tau)} \otimes_{W(\FF_{\fp(\tau)}),\tau} \FF = \prod_{\tau\in \Sigmabar} \FF[x]/(x^{e_{\fp(\tau)}}),
\end{align}
coming from~\eqref{eq:dec-O}. Note that the last equality of \eqref{eq:prod} depends on the choice of the uniformiser $\varpi_{\fp(\tau)}$ of $\fo_{\fp(\tau)}$, made in the previous subsection for every $\tau \in \Sigmabar$, and allows us to view $\omega_{\tau}$ as an $\cO_{X}[x]/(x^{e_\fp})$-module. If $S$ is a locally Noetherian $\FF$-scheme and $A$ is an HBAV over $S$ satisfying the Rapoport condition, then $\omega_{S,\tau}$ is a locally free $\cO_S[x]/(x^{e_{\fp(\tau)}})$-module of rank~$1$. Hence, there is a unique filtration on $\omega_{S,\tau}$ satisfying the Pappas--Rapoport conditions given by
$x^{e_{\fp(\tau)}-i}\omega_{S,\tau}$ for $0 \leqslant  i \leqslant  e_{\fp(\tau)}$. 
We point out again that the definition of $\cX$ depends on the numbering of the embeddings in $\Sigma_\tau$ fixed above,
but that $X$ is independent of any such choice (see also \cite[Rem.~2.3]{RX}).

If $\fp \mid p$ and $\tau \in \Sigmabar_{\fp}$, then suppose the universal filtration on $\omega_\tau$ is given by $(\cF^i_\tau)_{1 \leqslant i \leqslant e_{\fp}}$.
We now recall Reduzzi--Xiao's constructions of generalised Hasse invariants $h_{\sigma}$ given in \cite{RX}.
Let $\fp \mid p$ and $\tau \in \Sigmabar_\fp$ and assume first that  $2\leqslant  i \leqslant e_\fp$.  There is a map $\cF^{i}_{\tau} \to \cF^{i-1}_{\tau}$ which sends a local section $z$ of $\cF^{i}_{\tau}$ to the section $x\cdot z$ of $\cF^{i-1}_{\tau}$, where the action of $x$ is given by  
$[\varpi_{\fp(\tau)}]$. Hence, we get a map $\cF^{i}_{\tau}/\cF^{i-1}_{\tau} \to \cF^{i-1}_{\tau}/\cF^{i-2}_{\tau}$ inducing a section $h_{\tau,i}= h_{\sigma_{\tau,i}}$ of $\omega_{{\tau,i-1}} \otimes \omega_{\tau,i}^{-1} $ over $X$.
This $h_{\sigma}$ is the {\em generalised Hasse invariant at $\sigma = \sigma_{\tau,i}$} (see \cite[Construction 3.3]{RX} and \cite[\S2.11]{ERX} for more details).
As $(\omega_{\tau})_{|X^\Ra}$ is a locally free sheaf over $\cO_{X^\Ra}[x]/(x^{e_\fp})$ of rank~$1$, we have $(\cF^{i}_{\tau})_{|X^\Ra} = (x^{e_\fp-i}\omega_{\tau})_{|X^\Ra}$.  It follows that $h_{\tau,i}$ is a nowhere vanishing section over 
 $X^\Ra$ and  multiplication by $h_{\tau,i}$ induces an isomorphism between $(\omega_{\tau,i})_{|X^\Ra}$ and $(\omega_{\tau,i-1})_{|X^\Ra}$.

For the case $i=1$, the {\em generalised Hasse invariant} $h_{\tau,1}$ is defined
as a global section  over $X$ of $\omega^{\otimes p}_{{\phi^{-1}\circ\tau,e_\fp}}\otimes\omega^{\otimes-1}_{{ \tau,1}}$
(see \cite[Construction 3.6]{RX} for more details). We let $h_{\tau} =\prod_{\sigma \in \Sigma_\tau} h_{\sigma}=
\prod_{i=1}^{e_\fp}h_{\tau,i}$. It is a modular form of weight $p\cdot \widetilde{\phi^{-1} \circ \tau} - \widetilde{\tau}$.

\begin{rem}\label{rem:Hasse}
Let $A$ be the universal abelian scheme over $X$ and $\Ver: A^{(p)} \to A$ be the Verschiebung morphism, where $A^{(p)} = A \times_{\FF,\phi} \FF$.  It induces maps $\omega_{\tau} \to \omega^{(p)}_{\phi^{-1}\circ\tau}$ and further $\cF^{e_\fp}_{\tau}/\cF^{e_\fp-1}_{\tau} \to (\cF^{e_\fp}_{\phi^{-1}\circ\tau}/\cF^{e_\fp-1}_{\phi^{-1}\circ\tau})^{(p)}$. Note that $\cF^{e_\fp}_{\tau}/\cF^{e_\fp-1}_{\tau} =\omega_{\tau, e_\fp}$, $(\cF^{e_\fp}_{\phi^{-1}\circ\tau}/\cF^{e_\fp-1}_{\phi^{-1}\circ\tau})^{(p)} =\omega^{\otimes p}_{ \phi^{-1}\circ\tau,e_\fp}$ and the resulting  section of 
$\omega^{\otimes p}_{\phi^{-1}\circ \tau,e_\fp} \otimes \omega^{\otimes-1}_{\tau, e_\fp}$ over $X$ 
is precisely given by $h_{\tau}$ (see \cite[Lem.~3.8]{RX}). Moreover, its restriction to $X_\FF^\Ra$ 
coincides with Andreatta--Goren's {\em partial Hasse invariant} constructed in \cite[Def.~7.12]{AG}.  
In particular, when $p$ is unramified in $F$, the generalised Hasse invariants constructed by Reduzzi--Xiao are the same as the partial Hasse invariants constructed by Andreatta--Goren.
\end{rem}

We will now determine the geometric $q$-expansions of these generalised Hasse invariants. We will mostly follow conventions of \cite[\S8]{dimdg}.
Let  $\infty_{\fc}$ be the standard infinity cusp whose  Tate object  is given by $(\mathbb{G}_m \otimes_{\ZZ}\fc^*)/q^{\fo}$ (see \cite[\S2.3]{DiWi}). 
Here $\fc^*=\fc^{-1}\fd^{-1}$.
Let $X^{\wedge}$  be the formal completion of a toroidal compactification of $X$ along the divisor at the cusp  $\infty_{\fc}$ (see \cite[Thm.~8.6]{dimdg}). By {\it loc.~cit.}, the pull back of $\omega$ to $X^{\wedge}$ is canonically isomorphic to $\cO_{X^{\wedge}}\otimes \fc$. Choosing an identification 
\begin{align}\label{eq:identification}
\FF\otimes\fc\xrightarrow{\sim} \FF\otimes\fo
\end{align}
one can canonically identify  $\omega_{\tau|X^{\wedge}}$ with   $\tau(\cO_{X^{\wedge}}\otimes \fo)=\cO_{X^{\wedge}}[x]/(x^{e_{\fp(\tau)}})$ (see \eqref{eq:prod}). 
A global section of $\omega_{\tau}$ over $X^{\wedge}$ is an element of 
\[ \left\{ \sum_{\xi \in \fc_+ \cup \{0\}}a_{\xi}q^{\xi} \;|\; a_{\xi} \in \FF[x]/(x^{e_{\fp(\tau)}}) \text{ and } a_{u^2 \xi} = \tau(u)a_{\xi}, \forall  u \in \fo^{\times}, 
u-1\in \fn\right\},\]
whereas a section $z$ of $\omega_{\tau,i}$ over $X^{\wedge}$ is an element of 
\[ \left\{ x^{e_{\fp(\tau)}-i}\cdot \sum_{\xi \in \fc_+ \cup \{0\}}b_{\xi}q^{\xi} \;|\; b_{\xi} \in \FF \text{ and } b_{u^2 \xi} = \tau(u)b_{\xi}, \forall  u \in \fo^{\times}, 
u-1\in \fn\right\}\]
whose $q$-expansion is given by $\sum_{\xi \in \fc_+ \cup \{0\}}b_{\xi}q^{\xi}$ with respect to the choice of basis of $\omega_{\tau,i|X^\wedge}$ corresponding to $x^{e_{\fp(\tau)}-i}$.

\begin{lem}\label{qexplem} Let  $\fp | p$, $\tau \in \Sigmabar_{\fp}$.  Then for every $1 \leqslant  i \leqslant  e_{\fp}$, the geometric $q$-expansion of the generalised Hasse invariant $h_{\tau,i}$ at $\infty_{\fc}$ is $1$.
 In particular, it does not vanish at any cusp.
\end{lem}

\begin{proof}  
When $i>1$, as  $x\cdot z$ is a  section of $\omega_{\tau,i-1}$ having by definition the same 
$q$-expansion, one concludes that $h_{\tau,i}$ has $q$-expansion  $1$, thus proving the claim in that case.  
In the remaining case of $i=1$, we observe that the $q$-expansion of $h_{\tau} = \prod_{i=1}^{e_\fp(\tau)} h_{\tau,i}$  at $\infty_{\fc}$ is $1$ by Remark~\ref{rem:Hasse} and \cite[Prop.~7.14]{AG}.
Hence the $q$-expansion of $h_{\tau,1}$ at $\infty_{\fc}$ is $1$.
 Finally, since the $h_{\tau,i}$ can be defined in any level, we deduce their non-vanishing at all cusps from the non-vanishing at~$\infty_\fc$.
\end{proof}

We now collect some properties of the generalised Hasse invariants that will be used in the sequel.
Let $Z_{\sigma}\subset X$ be  the  divisor  of $h_{\sigma}$ and, in order to shorten the notation, we let 
$Z_{\tau,i}=Z_{\sigma_{\tau,i}}$.

\begin{lem}\label{lem:complement}
The complement of $X^\Ra$ in $X$ coincides with $\bigcup_{\tau\in \Sigmabar}\bigcup_{i=2}^{e_{\fp(\tau)}} Z_{\tau,i}$. 
Moreover, for any $I\subseteq \Sigma$, the intersection $\bigcap_{\sigma\in I}  Z_{\sigma}$ is, either empty, or equidimensional of dimension $d-|I|$. In particular, the zero loci of two different generalised Hasse invariants do not have a common divisor. 
\end{lem}

\begin{proof}
The first claim has  been established in \cite[Prop.~2.13 (2)]{ERX}. 
For the second, if $\cap_{\sigma\in I}  Z_{\sigma}$ is non-empty, then  the tangent space computation in  \cite[Thm.~3.10]{RX}  ensures the correct dimension.
\end{proof}

\begin{rem}\label{nonemptyrem}
Diamond and Kassaei also prove Lemma~\ref{lem:complement} and obtain in addition the non-emptiness of the intersection (see \cite[Prop.~5.8]{DK2}). Here we sketch a  constructive proof, following ideas of Andreatta and Goren~\cite{AG-IMRN},
if $e_{\fp(\tau)}$ is odd for all $\tau \in \Sigmabar$. 
 
Let $A = E \otimes_{\ZZ}\fo^*$, where $E$ is a supersingular elliptic curve over $\FF$.
We see, as in  \cite[Proof of Thm.~10.1]{AG-IMRN}, that $\omega_{A,\tau} \simeq \FF[x]/(x^{e_{\fp(\tau)}})$ for all $\tau \in \Sigmabar$.
Let  $\Frob_A : A \to A^{(p)}$ be the Frobenius map and $H = \ker(\Frob_A)[\prod_{\fp \mid p}\fp^{[e_{\fp}/2]}]$. 
By imitating the calculations of \cite[\S8]{AG-IMRN} (more specifically \cite[Prop.~6.5, Lemmas 8.6, 8.9, Prop.~8.10]{AG-IMRN}), one sees that if $A^{(1)} = A/H$, then
\begin{align}\label{eq:alternate}
\omega_{A^{(1)},\tau} \simeq x^{[e_{\fp(\tau)}/2]}\cdot \FF[x]/(x^{e_{\fp(\tau)}}) \bigoplus x^{e_{\fp(\tau)}-[e_{\fp(\tau)}/2]} \cdot \FF[x]/(x^{e_{\fp(\tau)}}) \text{ for all } \tau \in \Sigmabar.
\end{align}

Note that $A^{(1)}$ is a $\fc'$-polarised HBAV over $\FF$ for some $\fc' \in \cC$.
Let $\fa \subset \fo$ be an ideal  relatively prime to $p$ such that $\fa\fc'$ and $\fc$ represent the same element in the narrow class group of $F$.
Let $H^{(1)}$ be an $\fo$-invariant subgroup scheme of $A^{(1)}[\fa]$ isomorphic to $\fo/\fa$ and let $A^{(2)} = A^{(1)}/H^{(1)}$.
By \cite[\S 1.9]{KL}, $A^{(2)}$ is a $\fc$-polarised HBAV over $\FF$ and since $\fa$ is relatively prime to $p$, we have $\omega_{A^{(2)}} = \omega_{A^{(1)}}$.
Endowing each $\omega_{A^{(2)},\tau}$ with the `alternating' filtration between the two summands in \eqref{eq:alternate} yields a point in $\bigcap_{\tau \in \Sigmabar} \bigcap_{i=2}^{e_{\fp(\tau)}} Z_{\tau,i}$, showing that the latter is non-empty. 

If $e_{\fp(\tau)}$ is odd, then the  filtration on $\omega_{A^{(2)},\tau}$ described above is unique. 
Moreover, as $A^{(2)}$ is supersingular (i.e. its $p$-torsion subgroup has no \'{e}tale component), the map $\omega_{A^{(2)},\tau} \to \omega_{A^{(2)},\phi^{-1} \circ \tau}$ induced by the Verschiebung morphism is the zero map.
Hence, we conclude, using the structure of $\omega_{A^{(2)},\tau}$ and the definition of the Hasse invariant $h_{\tau,1}$, that any such point also lies in $Z_{\tau,1}$. Thus, if $e_{\fp(\tau)}$ is odd for all $\tau \in \Sigmabar$, then we get a point in $\bigcap_{\tau \in \Sigmabar} \bigcap_{i=1}^{e_{\fp(\tau)}} Z_{\tau,i}$.
\end{rem}

We illustrate the weights of the generalised and partial Hasse invariants in Table~\ref{fig:weights}, where we let $\tau \in \Sigmabar$ and write $e = e_{\fp(\tau)}$ as abbreviation.
\begin{figure}[h!]\label{fig:weights}
\caption{Weights of Hasse invariants.}
\begin{longtable}{||p{1.5cm}||*{3}{p{.8cm}|}|*{5}{p{.8cm}|}|*{3}{p{.8cm}|}|}
\hline
 & \multicolumn{11}{|c||}{Weights}\\
\cline{2-12}
 & \multicolumn{3}{|c||}{$\phi^{-1}\circ\tau$} & \multicolumn{5}{|c||}{$\tau$} & \multicolumn{3}{|c||}{$\phi\circ\tau$} \\
\cline{2-12}
 & $\cdots$ & $e-1$ & $e$ & $1$ & $2$ & $\cdots$ & $e-1$ & $e$ & $1$ & $2$ & $\cdots$ \\
\hline
$h_{\phi^{-1}\circ\tau,e}$ && $1$ & $-1$ &&&&&&&&\\
$h_{\tau,1}$ && & $p$ & $-1$ &&&&&&& \\
$h_{\tau,2}$ && && $1$ & $-1$ &&&&&&\\
$\vdots$ && &&  & $\ddots$ & $\ddots$ &&&&&\\
$h_{\tau,{e-1}}$ && && && $1$ & $-1$ &&&&\\
$h_{\tau,e}$ && && &&& $1$ & $-1$ &&&\\
$h_{\phi\circ\tau,1}$ && &&&&&&$p$ & $-1$ &&\\
$h_{\phi\circ\tau,2}$ && &&&&&&&$1$ & $-1$ &\\
\hline
$h_\tau$ && & $p$ &&&&&$-1$&&  &\\
\hline
\end{longtable}
\end{figure}

One of the advantages of Definition~\ref{defi:HMF} is that it allows us to define  mod $p$ 
Hilbert modular forms in {\it any} weight $k=\sum_{\sigma \in \Sigma} k_\sigma \sigma\in \ZZ[\Sigma]$, while the definition in~\cite{AG} was {\it missing} some weights when 
$p$ ramifies in~$F$, namely theirs are indexed by $\Sigmabar$, instead of $\Sigma$. Indeed, the space of modular forms 
introduced by  Andreatta and Goren \cite[Prop.~5.5]{AG} is 
\begin{align}\label{eq:AG-MF}
 M_{\bar k}^\AG(\fc,\fn;\FF) = \rH^0(X^\Ra,\bigotimes_{\tau \in \Sigmabar} \omega_{\widetilde{\tau}}^{k_\tau}),
\text { where } \bar k = \sum_{\tau \in \Sigmabar}k_{\tau}\tau \in \ZZ[\Sigmabar]. 
\end{align}
We will denote by $S_{\bar k}^\AG(\fc,\fn;\FF)$ the subspace of $M_{\bar k}^\AG(\fc,\fn;\FF)$ consisting of cuspforms,  which are defined as modular forms such that the constant coefficient of the $q$-expansion at every cusp vanishes.
If $k = \sum_{\sigma \in \Sigma}k_{\sigma}\sigma \in \ZZ[\Sigma]$, then for every $\tau \in \Sigmabar$, let $k_\tau =  \sum_{\sigma \in \Sigma_\tau} k_\sigma$ and define $\bar k := \sum_{\tau \in \Sigmabar}k_{\tau} \tau \in \ZZ[\Sigmabar]$.
We let
\begin{align}\label{eq:Ht}
H_k^\RX =\prod_{\tau\in \Sigmabar}\prod_{i=2}^{e_{\fp(\tau)}} h_{\tau,i}^{\sum_{j=1}^{i-1}k_{{\tau,j}}}, 
\end{align}
where $k_{\tau,j}=k_{\sigma_{\tau,j}}$. 
In view of the table of weights of the generalised Hasse invariants, for every $\tau \in \Sigmabar$,
the $(\tau,i)$-component of the weight of $f/H_k^\RX$ is $0$ if $1 \leqslant i \leqslant e_{\fp(\tau)}-1$ and the $\tautilde=(\tau,e_{\fp(\tau)})$-component is
$k_\tau$.
Since $H_k^\RX$ is invertible on  $X^\Ra$, we obtain the following result.

\begin{lem}\label{lem:AGtoRX} The restriction from $X$ to $X^\Ra$ yields an  injection of $M_k^\Katz(\fc,\fn;\FF)$ into $M_{\bar k}^\AG(\fc,\fn;\FF)$ sending $f$ to $f / H_k^\RX$. 
\end{lem}

A converse is described in Lemma~\ref{lem:RXtoAG} below.

\subsection{Minimal weights}\label{sec:min}

We recall the notion of {\em minimal weight} of a mod $p$ Hilbert modular form.

\begin{defi}\label{defi:minweight}
We define the {\em minimal weight} of $0 \neq f\in M_k^\Katz(\fc,\fn;\FF)$ to be the unique weight $k'$ such that 
$f =g\cdot  \prod_{\sigma \in \Sigma} h_{\sigma}^{n_\sigma}$, where $g\in  M_{k'}^\Katz(\fc,\fn;\FF)$ and  the integers $(n_\sigma)_{\sigma \in \Sigma}$  are as large as  possible. 
\end{defi}

\begin{lem}
The notion of minimal weight is well defined.
\end{lem}

\begin{proof}
First note that $Z_\sigma$ is non-empty for every $\sigma \in \Sigma$.
Indeed, this follows from \cite[Cor.~5.7]{DK2}.
Alternatively, we have shown in Remark~\ref{nonemptyrem} that $Z_{\tau,i}$ is non-empty for every $\tau \in \Sigmabar$ and $2 \leq i \leq e_{\fp(\tau)}$.
Moreover, it is well known that the zero locus of $h_{\tau} = \prod_{i=1}^{e_{\fp(\tau)}}h_{\tau,i}$ in $X^\Ra$ is non-empty for every $\tau \in \Sigmabar$ (see \cite[Cor.~8.18]{AG}).
By Lemma~\ref{lem:complement}, the Hasse invariants $h_{\tau,i}$ with $\tau \in \Sigmabar$ and $2 \leq i \leq e_{\fp(\tau)}$ are invertible on $X^\Ra$.
Therefore, it follows that the divisor $Z_{\tau,1}$ is non-empty for every $\tau \in \Sigmabar$.

Recall from Lemma~\ref{lem:complement} that the zero loci of two different generalised Hasse invariants do not have a common divisor. Let $j_{\sigma}$ be the order of vanishing of a Hilbert modular form $f\neq 0$ on $Z_\sigma$.
So, if we divide $f$ by $\prod_{\sigma \in \Sigma} h_{\sigma}^{j_\sigma}$, we get the modular form $g$ needed in Definition~\ref{defi:minweight}.
Hence, it follows that the notion of minimal weight is indeed well defined (see also the proof of \cite[Thm.~8.19]{AG} and \cite[\S 8]{DK2}).
\end{proof}

\begin{rem}
When $p$ is unramified in $F$ ({\it i.e.},  $\Sigma=\Sigmabar$),  the notion of minimal weights from Definition~\ref{defi:minweight}
is the same as the one  introduced by Andreatta and Goren \cite[\S8.20]{AG}. On the other hand, when $\fp$ is ramified, multiplying $0 \neq f\in M_k^\Katz(\fc,\fn;\FF)$ having minimal weight $k$ with arbitrary powers of generalised Hasse invariants  ($h_{\tau,i}$ with $2\leqslant  i \leqslant e_\fp$) yields forms sharing the same ${\bar k}$ but whose weights are not minimal anymore. 
\end{rem}

In  \cite{DK, DK2}, Diamond and Kassaei 
define the {\em minimal cone}  by
\[ C^\minim = \left\{ \sum_{\tau \in \Sigmabar} \sum_{i=1}^{e_{\fp(\tau)}}k_{\tau,i} \sigma_{\tau,i} \in \QQ[\Sigma] \;\Big{|} \;\forall\, \tau \in \Sigmabar, \forall\, 1 \leqslant  i < e_{\fp(\tau)}, k_{\tau,i+1} \geqslant k_{\tau,i},  p k_{\tau,1} \geqslant k_{\phi^{-1}\circ\tau,e_{\fp(\tau)}}\right\}.\]
Regarding the minimal weights for Hilbert modular forms, Diamond and Kassaei prove the following result in \cite[Cor. 5.3]{DK}, when $p$ is unramified in $F$, and in \cite[Cor. 8.2]{DK2}, when $p$ is ramified in $F$. 

\begin{prop}[Diamond--Kassaei]\label{prop:DKmin}
The minimal weight of  $0 \neq f \in M_k^\Katz(\fc,\fn;\FF)$ belongs to $C^\minim$.
\end{prop}

The minimal weights allow us to further elaborate on the relation between the modular forms defined by Andreatta and Goren \cite{AG} and those of Definition~\ref{defi:HMF}.

\begin{lem}\label{lem:RXtoAG}
Let $\bar k \in \ZZ[\Sigmabar]$.
There is a finite subset $K \subset C^\minim$ such that for every $f \in M_{\bar k}^\AG(\fc,\fn;\FF)$, there is $k' \in K$, a modular form 
$g \in M_{k'}^\Katz(\fc,\fn;\FF)$ and a product of generalised Hasse invariants $H = \prod_{\tau \in \Sigmabar} \prod_{i=1}^{e_{\fp(\tau)}} h_{\tau,i}^{j_{\tau,i}}$ with $j_{\tau,i} \in \ZZ$ and $j_{\tau,1}\geqslant 0$,  such that the restriction to $X^\Ra$ of $g \cdot H$ equals~$f$.
In particular, $f$ and $g$ have the same geometric $q$-expansion at the cusp $\infty_\fc$. 
\end{lem}

\begin{proof} The result is trivial for $f=0$.
Seeing $0 \neq f\in M_{\bar k}^\AG(\fc,\fn;\FF)$  as a meromorphic section of the line bundle $\bigotimes_{\tau \in \Sigmabar}{\omega_{\tautilde}}^{ \otimes k_{\tau}}$ over~$X$, we let $j_{\tau,i}\in \ZZ$ be the order of vanishing of $f$ along the divisor $Z_{\tau,i}$ defined by the Hasse invariant $h_{\tau,i}$ for $\tau \in \Sigmabar$ and $1 \leqslant i \leqslant e_{\fp(\tau)}$. As $f$ is holomorphic on $X^\Ra$, which intersects every irreducible component of $Z_{\tau,1}$ non-trivially by Lemma~\ref{lem:complement}, we deduce that  $j_{\tau,1}\geqslant 0$. 
Dividing $f$ by $H = \prod_{\tau \in \Sigmabar} \prod_{i=1}^{e_{\fp(\tau)}} h_{\tau,i}^{j_{\tau,i}}$  yields a holomorphic section on all of $X$, {\it i.e.}, a Katz modular form $g$ in a weight  $k'$ which is by construction minimal, hence belongs to $C^\minim$ by Proposition~\ref{prop:DKmin}. 
As the $q$-expansions of all generalised Hasse invariants at the cusp $\infty_\fc$ equal~$1$ by Lemma~\ref{qexplem}, both $f$ and $g$ have the same $q$-expansion.

We next prove that given $\bar k$,   there are only finitely many $k' \in C^\minim$ that can appear for  non-zero modular forms in $M_{\bar k}^\AG(\fc,\fn;\FF)$ via the method in the previous paragraph. 
Since dividing by $h_{\tau,1}$ (for any $\tau \in \Sigmabar$) subtracts $(p-1)$ from the sum of the weights, whereas multiplying or dividing by $h_{\tau,i}$ for $\tau \in \Sigmabar$ and $2 \leqslant i \leqslant e_{\fp(\tau)}$ leaves that sum unchanged, we deduce that
$ \sum_{\sigma \in \Sigma} k'_\sigma \leqslant  \sum_{\tau \in \Sigmabar} k_\tau$. 
As in the language of \cite{DK2}, the minimal cone is contained in the standard cone, we have $k'_\sigma \geqslant 0$ for all $\sigma \in \Sigma$ and the claimed finiteness follows.
\end{proof}

The finiteness of $K$ in Lemma~\ref{lem:RXtoAG} yields the following result.
\begin{cor}
The $\FF$-vector space $M_{\bar k}^\AG(\fc,\fn;\FF)$ is finite dimensional.
\end{cor}

We now further use the work of Diamond and Kassaei to study the minimality of the weight for modular forms of parallel weight one.

\begin{cor}\label{cor:minfil1}
Suppose $f \in M_1^\Katz(\fc,\fn;\FF)$ is a non-zero Hilbert modular form and $k$ is its minimal weight. 
Then, for any prime $\fp\mid p$, either $k_{\sigma}=1$ for all $\sigma \in \Sigma_{\fp}$  (in that case, we say that the weight is {\em minimal at~$\fp$}), or  $k_{\sigma}=0$ for all $\sigma \in \Sigma_{\fp}$, the latter case being possible only if $(p-1)$ divides $e_{\fp}$. 
\end{cor}

\begin{proof}
By Proposition~\ref{prop:DKmin}, we know that $k \in C^\minim$. By definition of $C^\minim$ one has 
 $k_{\sigma} \geqslant  0$ for all $\sigma\in \Sigma$ and, moreover, if $k_{\sigma} = 0$ with  $\sigma\in \Sigma_\fp$ for some
$\fp\mid p$, then $k_{\sigma} = 0$ for all $\sigma\in \Sigma_\fp$. 

We assume for the rest of this proof that  $k_{\sigma} = 0$  for all $\sigma\in \Sigma_\fp$. 
Denote the weight of the Hasse invariant $h_{\tau,i}$ by $w_{\tau,i}$.
By the definition of the minimal weight, there exist integers $n_{\tau,i} \geqslant  0$ such that
\begin{align}\label{sumeq}
\sum_{\sigma \in \Sigma_\fp} \sigma= \sum_{\tau \in \Sigmabar_{\fp}}\sum_{i=1}^{e_{\fp}}n_{\tau,i}w_{\tau,i}.
\end{align}
From the description of $w_{\tau,i}$ (see Table~\ref{fig:weights}), it follows that for all $i\geqslant  2$ one has $n_{\tau,i}=n_{\tau,i-1}+1$ and furthermore $pn_{\tau,1} = n_{\phi^{-1}\circ \tau,1} +e_{\fp}$. It is then easy to find that 
$n_{\tau,1}=\frac{e_{\fp}}{p-1}$ for all $\tau\in \Sigmabar_{\fp}$, showing  that $p-1$ divides $e_{\fp}$. 
\end{proof}

The following result, the proof of which will be completed in the next subsection, shows that one can be more precise when restricting to cuspforms.  

\begin{prop}\label{prop:cusp-zero}
Let $\fp$ be a prime of $F$ dividing~$p$.
Let $k = \sum_{\sigma \in \Sigma} k_\sigma \sigma \in \ZZ[\Sigma]$ be a weight such that $k_\sigma = 0$ for all $\sigma \in \Sigma_\fp$.
Then $S_k^\Katz(\fc,\fn;\FF) = 0$. 
\end{prop}

\begin{proof} By Lemma~\ref{lem:AGtoRX}, $S_k^\Katz(\fc,\fn;\FF)$ injects into $S_{\bar k}^\AG(\fc,\fn;\FF)$, which is zero by Proposition~\ref{cor:Sknull}.
Alternatively, if there is a unique prime $\fp$ of $F$ dividing~$p$, then  $k=0$ and Koecher's principle applied to an embedding of the connected scheme $X$ in a toroidal compactification implies that $\rH^0(X,\cO_{X})$ consists only of forms which are constant, thus it does not contain any non-zero cuspforms.
\end{proof} 

\begin{cor}\label{cor:cusp-minfil1}
The weight of any non-zero parallel weight $1$ cuspform is minimal. 
\end{cor}

\begin{proof} 
Let $f$ be a non-zero cuspform of parallel weight $1$ and let $k$ be its minimal weight.
Suppose the minimal weight $k$ is not $\sum_{\sigma \in \Sigma} \sigma$. Then by Corollary~\ref{cor:minfil1}
we already know that there exists $\fp\mid p$ such that $k_\sigma=0$ for all $\sigma\in \Sigma_\fp$.
Moreover, as the generalised Hasse invariants do not vanish at any cusp (see Lemma~\ref{qexplem}), 
we have constructed a  \emph{non-zero cuspform} of weight $k$, 
contradicting  Proposition~\ref{prop:cusp-zero}. This proves the corollary.
\end{proof}

\begin{rem}
Confusion may arise from the fact that parallel weight $1$ forms in our sense have weight exponents $e_{\fp(\tau)}$ when seen as a modular form in $M_{\bar k}^\AG(\fc,\fn;\FF)$ as $\bigwedge^{e_{\fp(\tau)}} \omega_\tau \simeq \omega_{\tautilde}^{\otimes e_{\fp(\tau)}}$ for $\tau \in \Sigmabar$ over the Rapoport locus (see Lemma~\ref{lem:AGtoRX}).
\end{rem}

\subsection{Partial Frobenius operator}\label{sec:partial-frob}
Fix $\fc\in \cC$ and let $e \in \NN$ be such that $\fp^{ee_\fp} = (\alpha)$ with $\alpha \in \fo_+$ and $\alpha \equiv p^e  \equiv 1 \pmod{\fn}$. Also let $\beta \in \fo_+$ such that $p^e=\alpha \cdot \beta$.
In order to lighten notation, we let $Y=X^\Ra$ denote the Rapoport locus and let $s:\cA \to Y$ be the universal  $\fc$-polarised HBAV endowed with $\mu_{\fn}$-level structure.
Let  $\cA^{(p^e)} = \cA \times_{Y, \Fr^e} Y$ be the base change by the $e$-th power of absolute Frobenius $\Fr: Y \to Y$.
The $e$-th power of Verschiebung then defines an  isogeny over~$Y$
\[ \cA^{(p^e)} \xrightarrow{\Ver^e} \cA, \]
the kernel of which we denote by~$H$. It is a finite group scheme with an ${\fo}/(p^e)$-action.
Hence we can apply the Chinese remainder theorem to obtain the direct product decomposition
$H = H_\fp \times H_\fp'$, where $H_\fp = H[\alpha]$ is the $\fp$-component of $H$ and $H_\fp' = H[\beta]$ is the product of $\fp'$-components of $H$ for all $\fp' \neq \fp$ dividing $p$.
We now define the abelian variety
\[\cB = \cA^{(p^e)} / H_\fp', \]
through which $\Ver^{e}$ factors, leading to an  isogeny over~$Y$
\begin{align}\label{eq:triangle} 
\xymatrix@R=1em{
\cB \ar@{->}[rr]^{V_\cA} \ar@{->}[dr]^{t}   && \cA \ar@{->}[dl]_{s}. \\
& Y\\}
\end{align}

\begin{lem}\label{lem:polquo}
The abelian variety $\cB$ inherits a $\mu_{\fn}$-level structure and an $\alpha^{-1}\fc$-polarisation~$\lambda_\alpha$.
\end{lem}

\begin{proof} As $\cA^{(p^e)}\to \cB$ is a $p$-primary isogeny, the $\mu_{\fn}$-level structure on $\cA^{(p^e)}$ yields one on~$\cB$. 

Regarding the polarisation, following a suggestion of the referee (see also \cite[\S1.9]{KL}), we claim that the kernel of the composed isogeny 
\[ \delta: \cB \otimes_{\fo} \fc \xrightarrow{V_\cA \otimes 1} \cA \otimes_{\fo} \fc \xrightarrow{\lambda^{-1}} \cA^{\vee} \xrightarrow{V_\cA^{\vee}} \cB^{\vee}\] 
equals the $\alpha$-torsion of~$\cB \otimes_{\fo} \fc$, {\it i.e.},  $\ker(\delta)$ is $\alpha$-torsion and has the same order as $(\cB \otimes_{\fo} \fc)[\alpha]$. As the order of finite flat group schemes is locally constant, it suffices to check this point-wise on 
the ordinary locus of $Y$  which is dense.

As $\Ver^e$ is \'{e}tale  at an ordinary closed point $y \in Y$, its kernel is isomorphic to the constant group scheme  given by $\fo/p^e \fo$, whence $\ker(V_{\cA_y}) \simeq   \fo/\alpha \fo$. Consequently, the kernel of the dual isogeny $V_{\cA_y}^{\vee}$ is  isomorphic to the Cartier dual $\mu_{\alpha \fo}$  of $\fo/\alpha \fo$.  This gives us a short  exact sequence of finite flat commutative group schemes 
\[0 \to \fc/\alpha\fc\simeq\ker(V_{\cA_y}) \otimes_{\fo} \fc \hookrightarrow \ker(\delta_y) 
\xrightarrow{\lambda^{-1} \circ (V_{\cA_y} \otimes 1)} \ker(V_{\cA_y}^\vee)\simeq \mu_{\alpha \fo} \to 0. \]

As the connected-\'etale sequence of any finite flat group scheme over a perfect field splits (see \cite[\S 3.7]{tate-FLT}), we deduce that $\ker(\delta_y)$  is isomorphic to the group scheme $(\fc/\alpha\fc)\times \mu_{\alpha \fo}$. In particular $\ker(\delta_y)$ is  $\alpha$-torsion and has the same order as $(\cB_y \otimes_{\fo} \fc)[\alpha]$. 
Therefore, it follows that  $\ker(\delta) = (\cB \otimes_{\fo} \fc)[\alpha]$. 

Since $(\cB \otimes_{\fo} \fc)/(\cB \otimes_{\fo} \fc)[\alpha]$ is canonically isomorphic to $\cB \otimes_{\fo} (\alpha^{-1} \fc)$, we deduce an  
isomorphism $\cB \otimes_{\fo}  (\alpha^{-1} \fc) \xrightarrow{\sim} \cB^\vee$ the inverse of which is the desired $\alpha^{-1}\fc$-polarisation~$\lambda_\alpha$.
\end{proof}

We now verify that the HBAV $\cB/Y$ satisfies the Rapoport condition. Recall that $\omega_{\cA / Y} = s_* \Omega^1_{\cA / Y}$ and $\omega_{\cB/Y}:=t_* \Omega^1_{\cB / Y}$.
\begin{lem}\label{lem:iso-away}
For any  $\tau \in \overline{\Sigma} \setminus\overline{\Sigma}_\fp$,  the map $V_{\cA,\tau}^*: \omega_{\cA / Y,\tau} \to \omega_{\cB / Y,\tau}$ is an isomorphism.

On the other hand, if $\tau \in \overline{\Sigma}_{\fp}$, then the isogeny $\cA^{(p^e)} \to \cB$ induces an isomorphism 
\[\omega_{\cB/Y,\tau} \simeq \omega_{\cA^{(p^e)}/Y,\tau} \simeq(\Fr^e)^*\omega_{\cA/Y,\phi^{-e}\circ \tau}.\]
\end{lem}

\begin{proof}
Let $r_{\fp}$ be the projection of $\fo/(p^e)$ on its  $\fp$-primary component and let $\gamma \in \fo$ be such that its image in $\fo/(p^e)$ represents $r_{\fp}$.
The the image of  $\gamma' =1-\gamma $  in $\fo/(p^e)$ represents the complementary idempotent
$r_\fp' = 1 - r_\fp$. 
As $\gamma'$ kills $\ker(V_\cA)$, the isogeny  $ \cB \xrightarrow{ \gamma'\cdot} \cB$ factors through~$V_\cA$,  yielding a factorisation
\[\xymatrix{
\omega_{\cB/Y} \ar[r] \ar@/^1.5pc/@{->}[rr]^{ \gamma'\cdot} & \omega_{\cA/Y} \ar@{->}[r]^{V_{\cA}^*} & \omega_{\cB/Y}
}\]
If $\fp' \mid p$ and $\fp' \neq \fp$, the projection of $\gamma'$ on the $\fp'$-component of $\fo/(p^e)$ is $1$. Hence, it induces  the identity  on the $\fp'$-component of $\omega_{\cB/Y}$.
So the map $V_{\cA}^*$ is split on the $\fp'$-component and hence $\omega_{\cB/Y,\tau}$ is isomorphic to a direct summand of $\omega_{\cA/Y,\tau}$ for all $\tau \in \overline{\Sigma}_{\fp'}$.
Recall that both $\omega_{\cB/Y,\tau}$ and $\omega_{\cA/Y,\tau}$ are locally free sheaves over $Y$ of the same rank.
 Therefore, after passing to their stalks, we conclude that $V_{\cA,\tau}^*$ is an isomorphism for all $\tau \in \overline{\Sigma}_{\fp'}$.

Similarly, as $\gamma$ annihilates the kernel of the isogeny $\cA^{(p^e)} \to \cB$,  we obtain  
an isomorphism between $\omega_{\cA^{(p^e)}/Y,\tau}$ and $\omega_{\cB/Y,\tau}$ for all $\tau \in \overline{\Sigma}_{\fp}$.
This proves the lemma.
\end{proof}

We get a $\fc$-polarisation on $\cB$ from the $\alpha^{-1}\fc$-polarisation $\lambda_{\alpha}$ (which is obtained in Lemma~\ref{lem:polquo})  after identifying  $\alpha^{-1}\fc$ with $\fc$ by multiplication by $\alpha$.
Thus, using Lemma~\ref{lem:iso-away},  the universal property of $\cA \to Y$ yields a Cartesian diagram
\begin{align}\label{eq:BAX}
\xymatrix@R=0.5em{
\cB \ar@{->}[rr] \ar@{->}[dd]^{t}   && \cA \ar@{->}[dd]^{s}, \\
&\Box & \\
Y\ar@{->}[rr]^{\phi_{\alpha}}  && Y,\\}
\end{align}
from which we deduce a natural isomorphism  $\phi_{\alpha}^* \omega_{\cA/Y} \xrightarrow{\sim} \omega_{\cB/Y}$ of $\fo \otimes \cO_{Y}$-modules.

Let $k = \sum_{\sigma \in \Sigma}k_\sigma \sigma \in \ZZ[\Sigma]$ be a weight such that $k_\sigma = 0$ for all $\sigma \in \Sigma_\fp$.
By  Lemma~\ref{lem:iso-away} 
\begin{align}\label{eq:adj}
 V_{\cA}^*: \omega_{\cA/Y}^{\otimes k} \xrightarrow{\sim} \omega_{\cB/Y}^{\otimes  k}.
\end{align}

\begin{defi}\label{defi:Phipe}
Let $k = \sum_{\sigma \in \Sigma}k_\sigma \sigma \in \ZZ[\Sigma]$ be a weight such that $k_\sigma = 0$ for all $\sigma \in \Sigma_\fp$.
The partial Frobenius  operator $\Phi_{\fp^e}$ is defined as the composition of the  adjunction morphism coming from \eqref{eq:BAX} with \eqref{eq:adj} 
\[\Phi_{\fp^e} : \rH^0(Y, \omega_{\cA/Y}^{\otimes k}) \xrightarrow{\phi_{\alpha}^*} \rH^0(Y, \omega_{\cB/Y}^{\otimes k})
\underset{\sim}{\xrightarrow{(V_{\cA}^*)^{-1}}}  \rH^0(Y, \omega_{\cA/Y}^{\otimes k}).\]
\end{defi}

We next study the effect of $\Phi_{\fp^e}$ on $q$-expansions. To this end, we recall the definition and properties of Tate objects.
For fractional ideals $\fa$, $\fb$, $\fc$ of $\fo$ such that $\fa\fb\subset\fc$ and a  cone $C$ in $\fc^*_+$, we let 
\[T_{\fa,\fb} = (\GG_m \otimes \fa^*) / q^\fb \to \overline{S}_C^{\circ}= \Spec(R^\circ_C)\]
be the Tate HBAV over the Noetherian algebra $R^\circ_C\supset \FF[[q^{\xi}, \xi \in \fc_+]]$
(for more details we refer to  \cite[\S2]{dimdg}, where  $R^\circ_C$ is denoted by  $R_C^\wedge\otimes_{R_C} R$).
It is equipped with a $\mu_{\fn}$-level  structure which depends on the choice of an isomorphism between $\fa/\fn\fa$ and $\fo/\fn$. Moreover, 
 the  natural isomorphism 
 \[\lambda_{\fa,\fb}: T_{\fa,\fb}^\vee=T_{\fb,\fa} \to T_{\fa,\fb} \otimes_{\fo} (\fa\fb^{-1})\]
endows $T_{\fa,\fb}$ with a canonical $\fa\fb^{-1}$-polarisation. Note that 
$ T_{\fc,\fo} = (\GG_m \otimes \fc^*) / q^\fo $  is a Tate HBAV  at  the standard cusp $\infty_{\fc}$ of $Y$ fitting, by  
 universality of $\cA/Y$, into a Cartesian diagram
\begin{align}\label{eq:cusp}
 \xymatrix@R=0.5em{
T_{\fc,\fo} \ar@{->}[rr]\ar@{->}[dd]   && \cA \ar@{->}[dd]^{s} \\
&\Box & \\
\overline{S}_C^{\circ}\ar@{->}[rr]^{\alpha_Y}  && Y.\\}
\end{align}
This gives a natural isomorphism $\alpha_Y^* \omega_{\cA/Y} \simeq \omega_{T_{\fc,\fo}/\overline{S}_C^{\circ}}$
and further, by adjunction and  choice of canonical trivialisations using \eqref{eq:identification}, we obtain a $q$-expansion map at the cusp $\infty_{\fc}$:
\[\rH^0(Y, \omega_{\cA/Y}^{\otimes k}) \xrightarrow{\alpha_Y^*} \rH^0(\overline{S}_C^{\circ}, \omega_{T_{\fc,\fo}/\overline{S}_C^{\circ}}^{\otimes k}) \simeq R^\circ_C.\]

Next we describe $\Ver^e$ on Tate objects.
Define $T_{\fc,\fo}^{(p^e)} = T_{\fc,\fo} \times_{\overline{S}_C^{\circ},\Fr^e} \overline{S}_C^{\circ}$ as the base change by the $e$-th power of absolute Frobenius. Note that $T_{\fc,\fo}^{(p^e)} = T_{p^e \fc,\fo}$ and the relative Frobenius map
$\Frob_{T_{\fc,\fo}}^e : T_{\fc,\fo} \to T_{p^e \fc,\fo}$ is the  map induced by the inclusion $p^e \fc \hookrightarrow \fc$.
The $p^e$-th Verschiebung is the dual of the $p^e$-th relative Frobenius on $T_{\fc,\fo}^{\vee}$, {\it i.e.}, $\Ver^e = (\Frob_{T_{\fc,\fo}^{\vee}}^e)^{\vee}$. We do not identify $((T_{\fc,\fo}^{\vee})^{(p^e)})^{\vee}$ with $T_{\fc,\fo}^{(p^e)}$ (as is usually done while defining Verschiebung) in order to get the desired maps on Tate objects. In particular, $\Frob_{T_{\fc,\fo}^{\vee}}^e : T_{\fo,\fc} \to T_{p^e\fo,\fc}$ is the map induced by the inclusion $p^e\fo \hookrightarrow \fo$. Therefore, its dual map $\Ver^e : T_{\fc,p^e\fo} \to T_{\fc,\fo}$ is the natural projection obtained by going modulo $q^{\fo}$.

Our next aim is to specialise $\phi_{\alpha}$ to the Tate objects. It follows from the previous paragraph that  the base change of \eqref{eq:triangle} to $\overline{S}_C^{\circ}$ is given by the following commutative diagram:
\[
\xymatrix@R=1em{
T_{\fc,\alpha\fo} \ar@{->}[rr]^{V_T} \ar@{->}[dr]  &  &  T_{\fc,\fo}  \ar@{->}[dl]. \\
& \overline{S}_C^{\circ} \\}
\]
where the map $V_T$ is the natural projection obtained by going modulo $q^{\fo}$. 
Combining  with \eqref{eq:BAX}, we get the following Cartesian diagram:
\begin{align}\label{eq:tate1}
\xymatrix@R=0.5em{
T_{\fc,\alpha\fo} \ar@{->}[rr] \ar@{->}[dd]   && \cB \ar@{->}[rr] \ar@{->}[dd]^{t} && \cA \ar@{->}[dd]^{s} \\
&\Box && \Box &\\
\overline{S}_{ C}^{\circ} \ar@{->}[rr]^{\alpha_Y}  && Y\ar@{->}[rr]^{\phi_{\alpha} } && Y.\\}
\end{align}

On the other hand, considering  the $\fc$-polarized HBAV $T_{\fc,\fo}$  over $\overline{S}_{\alpha C}^{\circ}$  gives a Cartesian diagram
\begin{align}\label{eq:tate2}
\xymatrix@R=0.5em{
T_{\fc,\alpha\fo} \ar@{->}[rr] \ar@{->}[dd]   && T_{\fc,\fo} \ar@{->}[rr]\ar@{->}[dd]   && \cA \ar@{->}[dd]^{s} \\
&\Box && \Box & \\
\overline{S}_{ C}^{\circ} \ar@{->}[rr]^{f_{\alpha}}  && \overline{S}_{\alpha C}^{\circ} \ar@{->}[rr]^{\alpha'_Y}  && Y, \\}
\end{align}
where $f_{\alpha}$ is induced by the morphism $R^\circ_{\alpha C} \to R^\circ_{C}$ sending $q^{\xi}$ to $q^{\alpha\xi}$.
We would like to emphasise  that $\alpha C$ is considered as a  cone in $\fc^*_+$, hence the dual cone (used in 
the construction of $R^\circ_{\alpha C}$, see  \cite[\S2]{dimdg}) is considered as a cone in $\fc$ (and not in $\alpha^{-1}\fc$).
In particular, the morphism $R^\circ_{\alpha C} \to R^\circ_{C}, q^{\xi}\mapsto q^{\alpha\xi}$ is not \'etale.

\begin{lem}\label{lem:phisC}
Under the notation developed above, $\phi_\alpha \circ \alpha_Y = \alpha'_Y \circ f_{\alpha}$.
\end{lem}

\begin{proof}
The proof proceeds by showing that the $\fc$-polarisation and $\mu_{\fn}$-level structure on $T_{\fc,\alpha\fo}$ obtained from the base change in ~\eqref{eq:tate1} coincide with the ones obtained from the base change in~\eqref{eq:tate2}. This, along with the universality of $\cA/Y$, implies $\phi_\alpha \circ \alpha_Y = \alpha'_Y \circ f_{\alpha}$.

As Mumford's construction of Tate objects presented in \cite[\S2]{dimdg} is functorial in $(\fa,\fb,\fc)$ and $C$, we deduce that 
the $\fc$-polarisation on $T_{\fc,\alpha\fo}$ arising from \eqref{eq:tate2} is obtained from $\lambda_{\fc,\alpha\fo}$ after identifying  $\alpha^{-1}\fc$ with $\fc$ by multiplication by $\alpha$.

We now derive the $\fc$-polarisation on $T_{\fc,\alpha\fo}$ via the base change in~\eqref{eq:tate1}. 
To do this, we proceed as in the proof of Lemma~\ref{lem:polquo} to first obtain a $\alpha^{-1}\fc$-polarisation on $T_{\fc,\alpha\fo}$ from $\lambda_{\fc,\fo}$. From the proof of Lemma~\ref{lem:polquo}, it follows that the kernel of the isogeny
\[T_{\fc,\alpha\fo} \otimes_{\fo} \fc \to T_{\fc,\fo} \otimes_{\fo} \fc \to T_{\fo, \fc} \to T_{\alpha\fo,\fc}.\]
is just the $\alpha$-torsion of $T_{\fc,\alpha\fo} \otimes_{\fo} \fc$.
Here the first map is induced by $V_T$ (the natural projection given by going modulo $q^{\fo}$), the second map is $\lambda_{\fc,\fo}^{-1}$, and the final map is induced by the inclusion $\alpha\fo \subset \fo$.
Therefore, this composition of maps induces an isomorphism 
\[\lambda : T_{\fc,\alpha\fo}^\vee=T_{\alpha\fo,\fc} \xrightarrow{\sim} \left(T_{\fc,\alpha\fo} \otimes_{\fo} \fc \right)\otimes_{\fo} \alpha^{-1}\fo = T_{\fc,\alpha\fo} \otimes_{\fo} \alpha^{-1}\fc,\]
which is the $\alpha^{-1}\fc$-polarisation on $T_{\fc,\alpha\fo}$ induced from $\lambda_{\fc,\fo}$.
From the description of the maps above, it follows that $\lambda=\lambda_{\fc,\alpha\fo}$.
Hence, the $\fc$-polarisation on $T_{\fc,\alpha\fo}$ via the base change in ~\eqref{eq:tate1} is obtained from $\lambda_{\fc,\alpha\fo}$ by identifying $\alpha^{-1}\fc$ with $\fc$ by multiplication by $\alpha$.
Therefore, it follows that these two $\fc$-polarisations on $T_{\fc,\alpha\fo}$ coincide.

As $\alpha \equiv 1 \pmod{\fn}$, the multiplication by $\alpha$ map preserves the $\mu_{\fn}$-level structure of $T_{\fc,\fo}$. 
Hence, the $\mu_{\fn}$-level structure on $T_{\fc,\alpha\fo}$ induced by $f_{\alpha}$ is same as the one coming from the quotient map $T_{\fc,p^e\fo} \to T_{\fc,\alpha\fo}$. This concludes the proof of the lemma.
\end{proof}

We are now ready to compute the effect of $\Phi_{\fp^e}$ on $q$-expansions at $\infty_{\fc}$.

\begin{prop}\label{cor:Sknull}
 Let $\bar k=\sum_{\tau \in \Sigmabar} k_\tau \tau \in \ZZ[\Sigmabar]$ such that $k_\tau=0$ for all $\tau \in \Sigmabar_\fp$.
Then the map $\Phi_{\fp^e}$ defines an endomorphism of $M_{\bar k}^\AG(\fc,\fn;\FF)$, sending $f=\sum_{\xi\in \fc_+}a_{\xi}q^{\xi}$ to $\Phi_{\fp^e}(f)=\sum_{\xi\in \fc_+} a_{\xi}q^{\alpha\xi}$.
In particular, the restriction of $\Phi_{\fp^e}$ to $S_{\bar k}^\AG(\fc,\fn;\FF)$ is injective and nilpotent, hence $S_{\bar k}^\AG(\fc,\fn;\FF)=\{0\}$.
\end{prop}

\begin{proof} By definition, the $q$-expansion of $\Phi_{\fp^e}(f)$ at  $\infty_\fc$ is the image of $f$ under the map 
$\rH^0(Y, \omega_{\cA/Y}^{\otimes k}) \to \rH^0(\overline{S}_{ C}^{\circ}, \omega_{T_{\fc,\alpha\fo}/\overline{S}_{ C}^{\circ}}^{\otimes k})$ 
coming from the Cartesian diagram \eqref{eq:tate1}, followed by $(V_T^*)^{-1}$.  By  Lemma~\ref{lem:phisC}, one can use instead the Cartesian diagram \eqref{eq:tate2}. 
Hence, the $q$-expansion of $\Phi_{\fp^e}(f)$ at  $\infty_\fc$  can be obtained as  the image of $f$ under the adjunction morphism 
\[\rH^0(Y, \omega_{\cA/Y}^{\otimes k}) \xrightarrow{\alpha'_Y{}^*} 
\rH^0(\overline{S}_{\alpha C}^{\circ}, \omega_{T_{\fc,\fo}/\overline{S}_{\alpha C}^{\circ}}^{\otimes k})  \xrightarrow{f_{\alpha}^*} 
 \rH^0(\overline{S}_{ C}^{\circ}, \omega_{T_{\fc,\alpha\fo}/\overline{S}_{ C}^{\circ}}^{\otimes k}) \]
followed by the map 
$(V_T^*)^{-1}:  \rH^0(\overline{S}_{ C}^{\circ}, \omega_{T_{\fc,\alpha\fo}/\overline{S}_{ C}^{\circ}}^{\otimes k}) \xrightarrow{\sim} \rH^0(\overline{S}_{ C}^{\circ}, \omega_{T_{\fc,\fo}/\overline{S}_{ C}^{\circ}}^{\otimes k})$. 
Since the $q$-expansion $\sum_{\xi\in \fc_+} a_{\xi}q^{\xi}$ of $f$ at the cusp $\infty_\fc$ is independent of a particular choice of a cone, it is given by the image of $f$ under the map $\alpha'_Y{}^* : \rH^0(Y, \omega_{\cA/Y}^{\otimes k}) \to \rH^0(\overline{S}_{\alpha C}^{\circ}, \omega_{T_{\fc,\fo}/\overline{S}_{\alpha C}^{\circ}}^{\otimes k})$. As  $f_{\alpha}$ is induced by the map sending $q^{\xi}$ to $q^{\alpha\xi}$ we deduce that  $f_{\alpha}^*(\sum_{\xi\in \fc_+} a_{\xi}q^{\xi}) = \sum_{\xi\in \fc_+}  a_{\xi}q^{\alpha \xi}$. Finally, 
as $V_T$ is induced from the identity map on the torus $\GG_m \otimes \fc^*$, the morphism $V_T^*$  is the identity on the $q$-expansions,
{\it i.e.},  $(V_T^*)^{-1}(\sum_{\xi\in \fc_+} a_{\xi}q^{\alpha\xi}) = \sum_{\xi\in \fc_+} a_{\xi}q^{\alpha\xi}$, yielding the desired formula. 

The rest  follows from  the $q$-expansion principle and the finite dimensionality of $S_{\bar k}^\AG(\fc,\fn;\FF)$. 
\end{proof}

\subsection{Refined injectivity criterion for $\Theta$-operators} \label{sec:theta}

The purpose of this section is to extend the definition of the Andreatta--Goren operators $\Theta^\AG_{\tau}$ for $\tau \in \Sigmabar$ and prove an injectivity criterion refining \cite[Prop.~15.10]{AG} when $p$ ramifies in $F$. 
Given  $f \in M_k^\Katz(\fc,\fn;\FF)$, by Lemma~\ref{lem:AGtoRX} and the discussion after it, $\Theta^\AG_{\tau}\left(\frac{f}{H_k^\RX}\right)$  defines a meromorphic section over $X$, whose poles lie outside  $X^\Ra$.  A careful study of the order at these poles will  first show that multiplication  by $H_k^\RX$ leads to a holomorphic section and further allow us to   establish  Proposition~\ref{prop:inj} (injectivity criterion).
If $p$ is unramified in~$F$, our $\Theta$-operators coincide exactly with those of Andreatta--Goren, and in that case everything that we prove here has already been proved in \cite{AG} (see also \cite{DS}).

The construction of  $\Theta^\AG_{\tau}$  goes via the Kummer cover.
By definition, the ordinary locus  $X^{\ord}$ of $X^\Ra$ is endowed with a Galois cover
$X^{\ord}(\mu_{(p)}) \to X^{\ord}$ with group $({\fo}/(p))^{\times}$, where $X(\mu_{(p)})$ is the Deligne--Pappas moduli space of level $p\fn$. 
Taking the quotient by the $p$-Sylow subgroup yields a
cover $\pi : X^{\kum} \to X^{\ord}$ with  group $\prod_{\fp\mid p}({\fo}/\fp)^{\times}$, called the {\em Kummer cover}. Let $\widetilde{\pi} : \widetilde{X} \to X$ be the normal closure of $X$ in $X^{\kum}$.
 It can be  described  explicitly using the generalised Hasse invariants as follows. 
For $\tau \in \Sigmabar$, we  write $\fp = \fp(\tau)$,  $f=f_{\fp}$ and we  let
\begin{align}\label{eq:mftau}
H_{\tau} = \prod_{j=0}^{f-1}(h_{\phi^{-j}\circ\tau})^{p^j}.
\end{align}
It is a modular form of weight $(p^f-1)\widetilde{\tau}$  and $X^{\kum}$ is  
obtained by adjoining a $(p^{f}-1)$-th root $s_{\tau}$ of it for all  $\tau \in \Sigmabar$. 
The nowhere vanishing section $s_{\tau}$ provides a trivialisation of the line bundle $\pi^*\omega_{\widetilde\tau}$ over $X^{\kum}$  (see \cite[Def.~7.4]{AG}), and by definition of the normalisation, it also defines a section over $\widetilde{X}$ (see \cite[Prop.~7.9]{AG}).
As
\begin{align}\label{eq:Htau}
H_{ \phi^{-1}\circ\tau}^p = H_{\tau} \cdot (h_{\tau})^{p^{f}-1},
\end{align}
$H_{ \phi^{-1}\circ\tau}^p / H_{\tau}$ is a $(p^{f}-1)$-th power, this construction does not depend on the choice of~$\tau \in \Sigmabar_{\fp}$.

Next we describe the Kodaira--Spencer maps.
By \cite[Thm.~2.9]{RX} there is a  decomposition 
\begin{align}\label{eq:dec-Omega}
\Omega^1_{X/\FF} = \bigoplus_{\tau\in \Sigmabar} \Omega^1_{X/\FF,\tau}
\end{align}
where  each  $\Omega^1_{X/\FF,\tau} $ is endowed with a filtration 
whose successive subquotients  are naturally isomorphic to $\omega_{\tau,i}^{\otimes 2}$ with $1 \leqslant  i \leqslant  e_\fp$ for $\fp = \fp(\tau)$ in descending order,
{\it i.e.},  $\omega_{\widetilde\tau}^{\otimes 2}=\omega_{{\tau,e_\fp}}^{\otimes 2}$ is naturally a quotient. 
As the map $\pi : X^{\kum} \to X^{\ord}$ is \'etale, we have $\Omega^1_{X^{\kum}/\FF} = \pi^*\Omega^1_{X^{\ord}/\FF}$,
the elements of which we view as meromorphic sections of the sheaf $\widetilde{\pi}^*\Omega^1_{X/\FF}$ over~$\widetilde{X}$.
Given a section of $\widetilde{\pi}^*\Omega^1_{X/\FF}$,  we denote by a subscript  $\tau \in \Sigmabar$ its projection onto the $\tau$-component via \eqref{eq:dec-Omega}. 
Consider the  surjective map 
\[\KS_\tau : \widetilde{\pi}^*\Omega^1_{X/\FF,\tau} \to \widetilde{\pi}^*\omega^{\otimes 2}_{\widetilde\tau}.\]

\begin{defi} 
Let $k = \sum_{\sigma \in \Sigma} k_\sigma \in \ZZ[\Sigma]$ and $k_\tau = \sum_{\sigma \in \Sigma_\tau} k_\sigma$ for $\tau \in \Sigmabar$.
Let $f \in M_k^\Katz(\fc,\fn;\FF)$.
Recall that $f/H_k^\RX\in M_{\bar k}^\AG(\fc,\fn;\FF)$ (see Lemma~\ref{lem:AGtoRX}).
We put
\[H_k^\AG =\prod_{\tau\in \Sigmabar} s_{\tau}^{k_\tau} \text{,  and } H_{k} = H_k^\AG \cdot \pi^* (H_k^\RX).\]
Similarly to \cite[Def.~7.19]{AG}, we further put
\[r(f) = \pi^*\big(f/H_k^\RX\big)/H_k^\AG = \pi^*(f) / H_{k} \in \rH^0(X^{\kum},\cO_{X^{\kum}})\]
where we restricted $f$ and $H_k^\RX$ to $X^\ord$.
\end{defi}

\begin{defi}\label{defi:theta}
For $\tau \in \Sigmabar$, we define the {\em generalised $\Theta$-operator} acting on  $f \in M_k^\Katz(\fc,\fn;\FF)$ as
\[ \Theta_{\tau}(f) = \KS_\tau\big(d(r(f))_{\tau}\big) \cdot H_{k} \cdot \pi^*( h_{\tau} )=
H_k^\RX\cdot\Theta^\AG_{\tau}\left(\frac{f}{H_k^\RX}\right) ,\]
viewed as an element of $\rH^0(X^{\kum}, \pi^*\omega^{\otimes k'})$ where\\
\begin{enumerate}
\item If $\tau \neq \phi^{-1}\circ \tau$, then $k'_\sigma = \begin{cases}
k_\sigma+1 & \textnormal{ if } \sigma = \sigma_{\tau, e_{\fp(\tau)}} = \widetilde{\tau},\\
k_\sigma+p & \textnormal{ if } \sigma = \sigma_{\phi^{-1}\circ\tau, e_{\fp(\tau)}} = \widetilde{\phi^{-1} \circ \tau},\\
k_\sigma   & \textnormal{ otherwise.}
\end{cases}$\\~\\
\item If $\tau = \phi^{-1}\circ \tau$, then $k'_\sigma = \begin{cases}
k_\sigma+p+1 & \textnormal{ if } \sigma = \sigma_{\tau, e_{\fp(\tau)}} = \widetilde{\tau},\\
k_\sigma   & \textnormal{ otherwise.}
\end{cases}$
\end{enumerate}
\end{defi}

We will now prove that $\Theta_{\tau}(f)$ yields an element of  $M_{k'}^\Katz(\fc,\fn;\FF)$. 
In order to prove this, we proceed as in \cite{AG} to calculate the poles of $d(r(f))_{\tau}$ along the  divisors of the generalised Hasse invariants.

Using the trivialisation of the line bundles $\pi^*\omega_{\widetilde{\tau}}$ given by the sections $s_{\tau}$ and the generalised Hasse invariants $h_{\tau,i}$'s for $i>1$, we get trivialisations of $\pi^*\omega_{\tau,i}$ for all $\tau \in \Sigmabar$ and $1 \leqslant i \leqslant e_{\fp(\tau)}$.
Using these trivialisations we can view the pullbacks $\widetilde{\pi}^*h_{\tau,i}$ and $\widetilde{\pi}^*h_{\tau}$ of Hasse invariants as functions over $\widetilde{X}$ (see \cite[\S 12.32]{AG} for more details), whose  differentials are denoted by $d(h_{\tau,i})$ and $d(h_{\tau})$, respectively  
(viewed  as meromorphic sections of $\widetilde{\pi}^*\Omega^1_{X/\FF}$).

For $\tau' \in \Sigmabar$, we let $\widetilde{Z}$ be an irreducible component of the effective Weil divisor of $\widetilde{X}$ associated to $\widetilde{\pi}^*(h_{\tau'})$ (see \S\ref{subsec:Hasse}). 
From the construction of $\widetilde{X}$ and \cite[\S9.3, Prop. 9.4]{AG} (see also \cite[\S12.32]{AG}), we can choose a uniformiser  $\delta$ at the generic point of $\widetilde{Z}$ such that $\delta^{p^{f_{\fp(\tau')}}-1} = H_{\tau'}$ (see~\eqref{eq:mftau} for the definition of $H_{\tau'}$).
We fix this choice from now on and let  $v_{\delta}$ be the corresponding normalised discrete valuation.
For the sake of readability, we will often drop $\widetilde{\pi}^*$ from the notation when pulling back Hilbert modular forms, especially generalised Hasse invariants; for instance, we usually write $v_\delta(h_\tau)$ for $v_\delta(\widetilde{\pi}^*(h_\tau))$.

We will first calculate $v_{\delta}((d\delta)_{\tau})$, where $(d\delta)_{\tau}$ is viewed as a meromorphic section of $\widetilde{\pi}^*\Omega^1_{X/\FF,\tau}$ over $\widetilde{X}$. 
We also prove some complementary results which will be used in the proof of the injectivity criterion.

\begin{lem}\label{lem:unif}
\begin{enumerate}[(i)]
\item\label{lem:unif:b} Let $\tau \in \Sigmabar$ different from~$\tau'$. Then $(d\delta)_{\tau}=0$.
\item\label{lem:unif:c} There is a unique $1 \leqslant i_0 \leqslant e_{\fp(\tau')}$ (depending on $\widetilde{Z}$) such that $v_\delta(h_{\tau',i_0}) = p^{f_{\fp(\tau')}}-1$ and $v_\delta(h_{\tau',i}) = 0$ for all $i \neq i_0$. 
Moreover, $v_{\delta}(h_{\tau'})=p^{f_{\fp(\tau')}}-1$ and $v_{\delta}(h_{\tau})=0$ if $\tau \neq \tau'$.
\item\label{lem:unif:d} $v_\delta(d(h_{\tau',i})) \geqslant 0$ for all~$i$ and for $i_0$ found in \eqref{lem:unif:c} , $v_\delta(d(h_{\tau',i_0})) = 0$.
\item\label{lem:unif:f} $v_{\delta}(s_{\tau'})=1$, $v_{\delta}(s_{\tau})=p^{j}$ if $\tau=\phi^j \circ \tau'$ and $v_{\delta}(s_{\tau})=0$ if $\tau \neq \phi^j \circ \tau'$ for any integer $j$.
\item\label{lem:unif:h} $v_{\delta}((d\delta)_{\tau'})=2-p^{f_{\fp(\tau')}}$.
\item\label{lem:unif:g} $(d\delta)_{\tau'} = D + g \cdot (d(h_{\tau',i_0}))_{\tau'}$ where
$g= -\delta^{2-p^{f_{\fp(\tau')}}} \cdot \big(\prod_{j=1}^{f_{\fp(\tau')}-1}(h_{\phi^{-j}\circ\tau'})^{{p^j}}\big) \cdot \big(\prod_{j \neq i_0}h_{\tau',j}\big)$ and
$D$ is a meromorphic section of $\widetilde{\pi}^* \Omega^1_{X/\FF,\tau'}$ such that $v_{\delta}(D) \geqslant  0$.
\item\label{lem:unif:i} $\KS_{\tau'}(d(h_{\tau',e_{\fp(\tau')}})_{\tau'}) \neq 0$ and if $\widetilde{Z}$ is an irreducible component of the effective Weil divisor of $\widetilde{X}$ associated with $\widetilde{\pi}^*(h_{\tau',e_{\fp}(\tau')})$, then  $v_{\delta}(\KS_{\tau'}(d(h_{\tau',e_{\fp(\tau')}})_{\tau'})s_{\tau'}^{-2})=-2$.
\end{enumerate}
\end{lem}

\begin{proof}
Recall that we have chosen $\delta$ such that $\delta^{p^{f_{\fp(\tau')}}-1} = H_{\tau'}$.
Hence 
\[-\delta^{p^{f_{\fp(\tau')}}-2}d\delta 
=  \big(\prod_{j=1}^{f_{\fp(\tau')}-1}(h_{\phi^{-j}\circ\tau'})^{{p^j}}\big) \cdot d(h_{\tau'}).  \]
Since $X^\Ra$ is Zariski dense in~$X$, it follows from \cite[Lem.~12.34]{AG} that $(d(h_{\tau'}))_{\tau} = 0$ if $\tau \neq \tau'$.
Hence, we get $d\delta=(d\delta)_{\tau'}$, which implies that $(d\delta)_{\tau}=0$ if $\tau \neq \tau'$.
Now
\[d(h_{\tau'}) = d(\prod_{i=1}^{e_{\fp(\tau')}}h_{\tau',i})
= \sum_{i=1}^{e_{\fp(\tau')}}  \big(\prod_{j \neq i}h_{\tau',j}\big) \cdot  d(h_{\tau',i}) .\]
Since $\delta$ is a  uniformiser at the generic point of $\widetilde{Z}$, there is a unique $i_0$ such that $v_{\delta}(h_{\tau',i_0})>0$.
Note that $v_{\delta}(h_{\tau',i})=0$ for $i \neq i_0$ and $v_{\delta}(d(h_{\tau',i})) \geqslant  0$ for all~$i$.
Moreover, it follows from \cite[Thm.~3.10]{RX} that $v_{\delta}(d(h_{\tau',i_0}))=0$.
So \eqref{lem:unif:c} and \eqref{lem:unif:d}  follow from the discussion above. 
Combining this with \eqref{eq:Htau} gives \eqref{lem:unif:f}.
Hence, $v_{\delta}(d( h_{\tau'} )) = v_{\delta}(\prod_{i \neq i_0}h_{\tau',i}(d(h_{\tau',i_0})))=0$ from which \eqref{lem:unif:h} follows and combining this with \eqref{lem:unif:c} and \eqref{lem:unif:d} gives us \eqref{lem:unif:g}.

We will now prove Statement~\eqref{lem:unif:i}.
Recall that, by \cite[Thm 2.9]{RX}, $\Omega^1_{X/\FF,\tau'}$ admits a canonical filtration whose successive subquotients are (isomorphic to) $\omega^{\otimes 2}_{\tau',i}$ with $1 \leqslant i \leqslant e_{\fp(\tau')}$.
Recall that $\KS_{\tau'}$ is the surjective map from $\widetilde{\pi}^*\Omega^1_{X/\FF,\tau'}$ onto its first subquotient $\widetilde{\pi}^*\omega^{\otimes 2}_{\tau',e_{\fp(\tau')}}$.
On the other hand, by \cite[Thm. 3.10]{RX}, $\Omega^1_{Z_{\tau',e_{\fp(\tau')}}/\FF,\tau'}$ admits a canonical filtration whose successive subquotients are (isomorphic to) $\omega^{\otimes 2}_{\tau',i}$ with $1 \leqslant i < e_{\fp(\tau')}$.
Here $Z_{\tau',e_{\fp(\tau')}} \subset X$ is the divisor of $h_{\tau',e_{\fp(\tau')}}$.
Therefore, we conclude that $\KS_{\tau'}(d(h_{\tau',e_{\fp(\tau')}})_{\tau'}) \neq 0$.
Since $s_{\tau'}$ gives a trivialisation of the line bundle $\pi^*\omega_{\widetilde{\tau'}}$, $v_{\delta}(\KS_{\tau'}(d(h_{\tau',e_{\fp(\tau')}})_{\tau'})s_{\tau'}^{-2})$ is well defined.
We conclude $v_{\delta}(\KS_{\tau'}(d(h_{\tau',e_{\fp(\tau')}})_{\tau'})s_{\tau'}^{-2})=-2$ by combining \cite[Thm. 3.10]{RX} with \eqref{lem:unif:d} and \eqref{lem:unif:f} (see also \cite[Prop. 12.37]{AG}).
This concludes the proof of the lemma.
\end{proof}

In order to compute $v_{\delta}(d(r(f))_{\tau})$, it is sufficient to work in the discrete valuation ring 
obtained by localising at the  generic point of $\widetilde{Z}$. 
Letting $r(f)=\frac{u}{\delta^n}$, with $v_{\delta}(u)=0$, we have 
\begin{align}\label{eq:val}
d(r(f))_{\tau} = \frac{(du)_{\tau}}{\delta^n}-\frac{nu(d\delta)_{\tau}}{\delta^{n+1}}.
\end{align}

\begin{lem}\label{lem:val}
Let $\tau \in \Sigmabar$ and let $f \in M_k^\Katz(\fc,\fn;\FF)$. Then
\begin{enumerate}[(i)]
\item\label{lem:val:a} $v_{\delta}((du)_{\tau}) \geqslant  \inf \{0,v_{\delta}((d{\delta})_{\tau})\}$,
\item\label{lem:val:b} $v_{\delta}(d(r(f))_{\tau}) \geqslant  v_{\delta}(r(f))$,  if $\tau \neq \tau'$,
\item\label{lem:val:c} $v_{\delta}(d(r(f))_{\tau'}) \geqslant  v_{\delta}(r(f))-(p^{f_{\fp(\tau')}}-2)$,  if $p|v_{\delta}(r(f))$,
\item\label{lem:val:d} $v_{\delta}(d(r(f))_{\tau'}) = v_{\delta}(r(f)) - (p^{f_{\fp(\tau')}}-1)$,  if $p \nmid v_{\delta}(r(f))$.
\end{enumerate}
\end{lem}

\begin{proof}
Let $\tau \in \Sigmabar$. The proof of \eqref{lem:val:a} is similar to the proof of \cite[Prop.~12.35]{AG} and we reproduce parts of it here.
Let $B$ (resp. $\widetilde{B}$) be the local ring of the generic point of $\widetilde{\pi}(\widetilde{Z})$ (resp. of $\widetilde{Z}$). 
As in   \cite[Cor.~9.6]{AG}, we have $B \subset B^{\et} \subset \widetilde{B}$ where $B^{\et}$ is \'etale over $B$ and $\widetilde{B}=B^{\et}[\delta]$.
Writing $u = \sum_{h=0}^{p^{f_{\fp(\tau')}}-2}u_h\delta^h$ with $u_h \in B^{\et}$, we have
\[(du)_{\tau} = \sum_{h=0}^{p^{f_\fp(\tau')}-2}\big(\delta^h (du_h)_{\tau} + hu_h\delta^{h-1} (d\delta)_{\tau} \big).\]
Now $\delta^h(du_h)_{\tau}$ lies in $\widetilde{B} \otimes_{B^{\et}} \Omega^1_{B^{\et}/\FF}$, which is the same as $\widetilde{\pi}^*\Omega^1_{B/\FF}$ since $B^{\et}$ is \'etale over~$B$.
Hence $\delta^h(du_h)_{\tau}$ has no poles and \eqref{lem:val:a} follows.
Combining this inequality with \eqref{eq:val} and   Lemma~\ref{lem:unif}\eqref{lem:unif:b},~\eqref{lem:unif:h} gives us the other parts of the lemma.
\end{proof}

Finally we are ready to prove that $\Theta_{\tau}(f)$ is also a mod $p$ Hilbert modular form.

\begin{prop}\label{prop:descend}
Let  $\tau \in \Sigmabar$, $f \in M_k^\Katz(\fc,\fn;\FF)$ and  $k'$ be as in Definition~$\ref{defi:theta}$.
Then $\Theta_{\tau}(f)$ descends to a global section of the line bundle $\omega^{\otimes k'}$ over~$X^\ord$,
and further extends to a section over $X$, yielding an element  $\Theta_{\tau}(f) \in M_{k'}^\Katz(\fc,\fn;\FF)$. 
\end{prop}

\begin{proof}
The descent  follows by applying \cite[Thm.~12.39]{AG} to $f/H_k^\RX$.

As $\KS_{\tau}$ is a surjective map of locally free sheaves with a locally free kernel over the normal scheme $\widetilde{X}$, the orders of the poles of $\KS_{\tau}(d(r(f))_{\tau})$ are less than or equal to the orders of the poles of $d(r(f))_{\tau}$, {\it i.e.},
$v_\delta(\KS_{\tau}(d(r(f))_{\tau})) \geqslant v_\delta(d(r(f))_{\tau})$ (see the proof of \cite[Prop.~12.37]{AG} for more details). 
Note that $r(f) \cdot H_{k} = \pi^*(f)$ has no poles on~$\widetilde{X}$,
{\it i.e.}, $v_{\delta}(\pi^*f) \geqslant 0$.
Combining  Lemma~\ref{lem:unif}\eqref{lem:unif:c} and Lemma~\ref{lem:val}, we get that $\Theta_{\tau}(f)$ has no poles over $\widetilde{X}$.
Hence, the section obtained by descending from $X^{\kum}$ to $X^{\ord}$ extends to all of~$X$ and is thus a Hilbert modular form.
\end{proof}

The effect of $\Theta_{\tau}$ on the geometric $q$-expansions of Hilbert modular forms will be used in 
\S\ref{subsec:doubling} and can be described as follows.  The identification \eqref{eq:identification}, used in defining the geometric
$q$-expansion $\sum_{\xi \in \fc_+ \cup \{0\}}a_{\xi}(f)q^{\xi}$ of $f$ at the cusp $\infty_{\fc}$, allows one to consider the map
\[\bar\tau_\fc: \FF\otimes\fc\xrightarrow{\sim} \FF\otimes\fo \twoheadrightarrow \FF[x]/(x^{e_{\fp(\tau)}})\twoheadrightarrow \FF,  \]
where the middle map is given by the idempotent at~$\tau$. 
By \cite[Cor.~12.40]{AG} we obtain the following  $q$-expansion at the cusp $\infty_{\fc}$:
\begin{align}\label{eq:q-exp-theta}
\Theta_\tau(f) = \sum_{\xi \in \fc_+} \bar\tau_\fc(1\otimes\xi)a_{\xi}q^{\xi}. 
\end{align}

The proof of our main theorem uses the injectivity of $\Theta_{\tau}$ on certain mod~$p$ Hilbert modular forms.  
\begin{prop}\label{prop:inj}
Let $f \in M_k^\Katz(\fc,\fn;\FF)$ and let $\tau\in \Sigmabar_\fp$. 
Suppose $p \nmid k_{\tau,e_{\fp}}$ and $h_{\tau,e_{\fp}}$ does not divide $f$. Then $\Theta_{\tau}(f) \neq 0$.

In particular, if the weight of $f$ is minimal at $\fp$, and $p \nmid k_{\tau,e_{\fp}}$, then $\Theta_{\tau}(f) \neq 0$.
\end{prop}

\begin{proof} We follow the proof of  \cite[Prop.~15.10]{AG}. Let $\delta$ be the uniformiser at the generic point of an irreducible component of the Weil divisor of $\widetilde{X}$ attached to $\widetilde{\pi}^*h_{\tau,e_{\fp}}$ chosen just before Lemma~\ref{lem:unif}. 
As $v_{\delta}(\widetilde{\pi}^*(f))=0$, using Lemma~\ref{lem:unif}\eqref{lem:unif:c},\eqref{lem:unif:f} (see also \cite[Prop.~15.9]{AG})  we deduce
\[ n:=v_{\delta}(r(f)) = -\sum_{j=0}^{f-1}p^j k_{\phi^j\circ\tau}-(p^{f_{\fp}}-1)\sum_{j=1}^{e_{\fp}-1}k_{\tau,j}   \equiv -k_{\tau,e_{\fp}}\pmod{p}.\]
Hence  $p \nmid n $ and Lemma~\ref{lem:val}~\eqref{lem:val:a} shows that the right most term in~\eqref{eq:val} has a strictly lower valuation than the other term on the right hand side.
Thus, Lemma~\ref{lem:unif}~\eqref{lem:unif:g} shows that
\[d(r(f))_{\tau}=D' - \frac{n    u \delta^{2-p^{f_{\fp(\tau)}}}\cdot \big(\prod_{j=1}^{f_{\fp(\tau)}-1}(h_{\phi^j\circ\tau})^{{p^j}}\big) \cdot \big(\prod_{j \neq e_{\fp}}h_{\tau,j}\big)}{\delta^{n+1}} (dh_{\tau,e_{\fp}})_\tau,\]
where $D'$ is a meromorphic section of $(\widetilde{\pi}^*\Omega^1_{X/\FF})_{\tau}$ and the right most term has a strictly smaller valuation than $D'$.
Combining this with Lemma~\ref{lem:unif}~\eqref{lem:unif:i}, we get that $\KS_{\tau}(d(r(f))_{\tau}) \neq 0$.
This implies that $\Theta_{\tau}(f) \neq 0$.
\end{proof}

\begin{rem}
When $p$ is unramified in $F$, Proposition~\ref{prop:inj} can also be deduced from 
  \cite[Thm.~8.2.2]{DS} whose  proof is different. Furthermore,  in  \cite[Thm.~9.8.2]{DS}, Diamond and Sasaki also determine the kernel of $\Theta_\tau$ in terms of the partial Frobenius operator at $\tau$ that they define. 
Meanwhile, the case when $p$ is ramified in $F$ has been  treated in \cite{D}.
 Proposition~\ref{prop:inj} follows from \cite[Thm. 5.2.1]{D} and the kernel of $\Theta_{\tau}$ is described in terms of partial Frobenius operators in \cite[Cor. 9.1.2]{D}.
\end{rem}

\section{Doubling and Hecke algebras}\label{sec:doubling}

\subsection{Hilbert modular forms of parallel weight $1$}
It is important to distinguish between Katz Hilbert modular forms defined on the fine moduli space and those on the coarse quotient by the action of the totally positive units of~${\fo}$. The latter enjoy the good Hecke theory for $\GL(2)$ and are the natural objects to study
in relation with two dimensional Galois representations (see~\cite{DiWi}). 
In this section, we will define Hilbert modular forms of parallel weight building on Definition~\ref{defi:HMF}.
Even though we give a definition valid in all levels $\fn$ that are prime to~$p$, we nevertheless need to consider the following condition (which is stronger than the one we imposed in \S\ref{section1}) expressing that $\fn$ is sufficiently divisible:
\begin{align}\label{ass:21}
\begin{split}
& \text{$\fn$ is divisible by  a prime above a prime number $q$   splitting completely in $F(\sqrt{\epsilon} \,|\,
\epsilon\in{\fo}_+^\times)$}, \\
& \text{and such that  $q\equiv -1\pmod{4\ell}$ for all prime numbers $\ell$ such that
$[F(\mu_\ell):F]=2$. }
\end{split}
\end{align}
This condition ensures that $\cX^{\DP}$  is a scheme on which  $[\epsilon] \in E= {\fo}^{\times}_+/\{\epsilon \in {\fo}^{\times}|\epsilon-1 \in \fn\}^2$ acts properly and discontinuously  by sending $(A, \lambda, \mu)$ to $(A, \epsilon\lambda, \mu)$ (see    \cite[Lem.~2.1(iii)]{dimitrov:ihara}). For any  $\fc\in \cC$, any $\ZZ_p$-algebra $R$, and any parallel weight $k$,   this  induces an action of $E$ on   $M_k^\Katz(\fc,\fn;R)$, whose invariants are denoted by  $M_k^\Katz(\fc,\fn;R)^E$.   The following definition is equivalent to the one used in  \cite[\S2.2]{DiWi}. 

\begin{defi}\label{defi:HMFWiles}
If $\fn$ satisfies~\eqref{ass:21}, then the space of Hilbert modular forms over a $\ZZ_p$-algebra $R$ of (parallel) weight~$\kappa\in \ZZ$ and level~$\fn$ is given by 
\[M_\kappa(\fn,R)=\bigoplus_{\fc\in \cC} M_{k}^\Katz(\fc,\fn;R)^E,\]
where $k = \sum_{\sigma \in \Sigma}\kappa \sigma$.
For a general level~$\fn$, let $\fq_1 \neq \fq_2$ be primes such that both $\fn \fq_1$ and $\fn \fq_2$ satisfy~\eqref{ass:21} and define
\[M_\kappa(\fn,R)=M_\kappa(\fn \fq_1,R) \cap M_\kappa(\fn \fq_2,R),\]
where the intersection can be taken in $M_\kappa(\fn \fq_1\fq_2,R)$.
Note that the primes $\fq_1$ and $\fq_2$ can be chosen from a set of primes of positive density and that the definition does not depend on this choice.

For $f \in M_{\kappa}(\fn, R)$, we let  $\sum_{\fb \in \cI \cup \{(0)\}}a(\fb,f)q^{\fb}$ be the adelic $q$-expansion of~$f$, where 
$\cI$ denotes the group of fractional ideals of $F$ (see \cite[\S2.6]{DiWi}). 

We denote by $S_\kappa(\fn,R)$ the $R$-submodule of $M_\kappa(\fn,R)$ consisting of  Hilbert modular cuspforms.

\end{defi}

For $f\in M_\kappa(\fn,R)$ and $\fc\in \cC$, we will let  $f_\fc$ denote the corresponding $E$-invariant element of 
$M_k^\Katz(\fc,\fn;\FF)$ or, equivalently, its geometric $q$-expansion at the cusp $\infty_\fc$ (see \cite[\S2.5]{DiWi}). 
Recall that when $\fn$ satisfies~\eqref{ass:21}, $M_\kappa(\fn,R)$ is endowed  with  Hecke and diamond operators (see \cite[\S3.1-3.3]{DiWi}).
When $\fn$ is not sufficiently divisible, Hecke and diamond operators exist nonetheless because they stabilise the intersection $M_\kappa(\fn \fq_1,R) \cap M_\kappa(\fn \fq_2,R)$, where the auxiliary primes $\fq_1,\fq_2$ may be chosen appropriately.
When it is not clear from the context, a superscript between brackets indicates the weight of the space of Hilbert modular forms on which an operator acts, e.g.\ $T_\fp^{(1)}$.
Since we are interested in torsion coefficients,  we let $M_\kappa(\fn ,K/\cO) = \ilim{n} M_\kappa(\fn ,\cO/\varpi^n)$, where the inductive limit is taken by identifying $M_\kappa(\fn ,\cO/\varpi^n)$ with  the subspace $M_\kappa(\fn ,\cO/\varpi^n)\otimes_{\cO} (\varpi\cO)$
of $M_\kappa(\fn ,\cO/\varpi^{n+1})$. 

\subsection{Doubling}\label{subsec:doubling}

We shall rely on the following lifting result.

\begin{lem}\label{lem:lifting} Suppose that $\fn$ satisfies~\eqref{ass:21}.
There exists a $\kappa_0\in \ZZ$ such that for all  $\kappa\geqslant  \kappa_0$ and all $n\in\NN$, the natural map
\[M_\kappa(\fn ,\cO)  \otimes_{\cO}  \cO/\varpi^n  \to  M_\kappa(\fn, \cO/\varpi^n)\]
is a Hecke equivariant isomorphism. 
\end{lem}

\begin{proof}
The proof of \cite[Lem.~2.2]{DiWi} works unchanged after replacing $\ZZ_p$ by~$\cO$ and $p$ by $\varpi^n$.
\end{proof}

We also need a generalisation of the total Hasse invariant modulo~$\varpi^n$.

\begin{lem}\label{lem:Hasse}
For every $n \in \NN$, there is a $\kappa_n \in \NN$ such that $(\kappa_n -1)$ is a multiple of $(p-1)p^{n-1}$, 
and a modular form $h_n \in M_{\kappa_n-1}(\fo, \cO/\varpi^n)$ having $q$-expansion equal to~$1$ at $\infty_\fc$ for all $\fc\in \cC$.  In particular, it does not vanish at any cusp.
\end{lem}

\begin{proof}
Let $h \in M_{p-1}(\fo,\FF)$ be the usual Hasse invariant (see \cite[\S3.4]{DiWi}). Note that since it exists for every level satisfying~\eqref{ass:21}, it exists in level~$\fo$. For $r$ such that $r(p-1)$ is big enough to apply Lemma~\ref{lem:lifting}, the 
 modular form  $h^r\in M_{r(p-1)}(\fo,\cO)$ has $q$-expansion congruent to~$1$  modulo~$\varpi$ at each cusp $\infty_\fc$, $\fc\in \cC$.
 A big enough power of it satisfies the required congruence relation and condition on the weight.  
\end{proof}

The theory of generalised $\Theta$-operators presented in \S\ref{sec:theta} allows us to prove the following result. 

\begin{lem}\label{lem:onep}
Assume that $\fn$ satisfies~\eqref{ass:21}.
Then there does not exist any $0 \neq f \in M_1(\fn ,\FF)$ such that $f_\fc$ has minimal weight at a fixed prime $\fp$ dividing $p$ (see Corollary~\ref{cor:minfil1}) for all $\fc\in \cC$ and such that $a(\fb,f)=0$ for all ideals $\fb\subset {\fo}$ not divisible by~$\fp$.
\end{lem}

\begin{proof} 
The minimality of the weight at~$\fp$ implies that there exists a $\tau \in \Sigmabar_\fp$ such that $h_{\widetilde\tau}$ does not divide $f_\fc$ (the proof of Corollary~\ref{cor:minfil1} implies that this is true for all $\tau \in \Sigmabar_\fp$). 
Let $\fb = (\xi) \fc^{-1}$. Then, by definition, $a_\xi = a(\fb,f)$ 
and this is zero unless $\fp \mid \fb$, in which case $\fp \mid (\xi)$.
Thus, it follows that $\bar\tau_\fc(1\otimes\xi) = 0$.
By \eqref{eq:q-exp-theta}, this shows that the geometric $q$-expansion of~$\Theta_\tau(f_\fc)$ vanishes at  $\infty_{\fc}$ for all $\fc \in \cC$, {\it i.e.}, $\Theta_\tau(f_\fc)=0$, contradicting the injectivity criterion from Proposition~\ref{prop:inj}. 
\end{proof}

For $\fp \mid p$ and $n \in \NN$, we define the $V_\fp $-operator by (see \cite{DiKS}, improving on and correcting previous works such as \cite{ERX} and \cite{DiWi}, for the  definition of $T_\fp^{(1)}$)
\[  V_{\fp,n} = \diam{\fp}^{-1}(h_n T_\fp ^{(1)}  -  T_\fp ^{(\kappa_n)} h_n )  \] 
with $h_n$ and $\kappa_n$ from Lemma~\ref{lem:Hasse}.
A simple computation on $q$-expansions (see \cite[Prop.~3.6]{DiWi}) shows that $V_{\fp,n}$ has the following effect on adelic $q$-expansions:
\begin{align*}
a((0),V_{\fp,n}f)  &= a((0),f)[\fp^{-1}],\\
a(\fr, V_{\fp,n}f) &= a( \fp^{-1}\fr,f)
\end{align*}
for non-zero ideals $\fr \subseteq {\fo}$.

\begin{prop}\label{prop:doubling}
Let $\fp\mid p$ be a prime and assume that $\fn$ satisfies~\eqref{ass:21}.
\begin{enumerate}[(i)]
\item\label{item:doubling:a} If $f \in S_1(\fn ,K/\cO)$ and  $a(\fb,f)=0$ for all ideals $\fb\subset {\fo}$ not divisible by~$\fp$, then $f=0$. 
\item\label{item:doubling:b} For all $n \in \NN$, the `doubling map'
\[ (h_n,V_{\fp,n}): S_1(\fn ,\cO/\varpi^n)^{\oplus 2} \xrightarrow{(f,g) \mapsto h_n f + V_{\fp,n}g} M_{\kappa_n}(\fn ,\cO/\varpi^n)\]
is injective and compatible with the Hecke operators $T_\fq$ for $\fq \nmid \fn p$.
The Hecke operator $T_\fp ^{(\kappa_n)}$ acts on the image by the formula $T_\fp ^{(\kappa_n)}  \circ (h_n,V_{\fp,n})  = (h_n,V_{\fp,n}) \circ \left(\begin{smallmatrix} T_\fp^{(1)} & 1  \\ -\diam{\fp} & 0 \end{smallmatrix}\right)$.
In particular, the image $W_{\fp,n}$ of $(h_n,V_{\fp,n})$ lies in the $\fp$-ordinary part of $M_{\kappa_n}(\fn ,\cO/\varpi^n)$ and is stable under all Hecke operators $T_\fq$ for $\fq \nmid \fn p$.
\end{enumerate}
If $(p-1)$ does not divide $e_\fp$, then the same statements hold after replacing the spaces $S_1(\fn ,K/\cO)$ and $S_1(\fn ,\cO/\varpi^n)$ by $M_1(\fn ,K/\cO)$ and $M_1(\fn ,\cO/\varpi^n)$, respectively.
\end{prop}

\begin{proof}
\eqref{item:doubling:a}
For $f \in S_1(\fn ,\cO/\varpi)$, the claim is precisely the content of Lemma~\ref{lem:onep}, in view of Corollaries \ref{cor:minfil1} and~\ref{cor:cusp-minfil1}. 
The induction step from $n-1$ to~$n$ follows from the $q$-expansion principle and the exact sequence
\[ 0 \to S_1(\fn ,\cO/\varpi) \otimes_{\cO} \varpi^{n-1}\cO \to  S_1(\fn ,\cO/\varpi^n) \to S_1(\fn ,\cO/\varpi^{n-1}).\]
\eqref{item:doubling:b} The injectivity follows from \eqref{item:doubling:a} applied to the first component of an element in the kernel. 
The matrix is obtained from a calculation as in~\cite[Lem.~3.7]{DiWi}.
\end{proof}

\subsection{Hecke algebras}

For $\kappa\geqslant  1$ and $n\in \NN$, we consider the following  complete Artinian (resp.\ Noetherian) semi-local $\cO$-algebras
\begin{align}
\begin{split} 
\TT^{(\kappa)}_n & = \Image\big(\cO[T_\fq, \diam{\fq}]_{\fq \nmid \fn p} \to \End_\cO(M_\kappa(\fn,\cO/\varpi^n))\big),  \\
\TT^{(\kappa)}_{\cusp,n} & = \Image\big(\cO[T_\fq, \diam{\fq}]_{\fq \nmid \fn p} \to \End_\cO(S_\kappa(\fn,\cO/\varpi^{n}))\big), \text{ resp.,} \\
\TT^{(\kappa)} & =\Image\big(\cO[T_\fq, \diam{\fq}]_{\fq \nmid \fn p} \to \End_\cO(M_\kappa(\fn ,K/\cO))\big)=\varprojlim\limits_n \TT^{(\kappa)}_n,\\
\TT^{(\kappa)}_\cusp & =\Image\big(\cO[T_\fq, \diam{\fq}]_{\fq \nmid \fn p} \to \End_\cO(S_\kappa(\fn ,K/\cO))\big)=\varprojlim\limits_n \TT^{(\kappa)}_{\cusp,n}.
\end{split}
\end{align}
Note that they all contain $\diam{\fp}$ for $\fp \mid p$ since  $p$  is relatively prime to~$\fn$.
Moreover, the restriction to the cusp space gives surjective morphisms $\TT^{(\kappa)}_n \twoheadrightarrow \TT^{(\kappa)}_{\cusp,n}$ and $\TT^{(\kappa)}\twoheadrightarrow \TT^{(\kappa)}_\cusp$. We  also consider the torsion free Hecke $\cO$-algebra: 
\[\TT^{(\kappa)}_\cO= \Image\big(\cO[T_\fq, \diam{\fq}]_{\fq \nmid \fn p} \to \End_\cO(M_\kappa(\fn ,\cO))\big). \]

 Let $I_n$ be the annihilator of $\TT^{(\kappa)}_\cO$ acting on $M_\kappa(\fn ,\cO) \otimes_\cO (\cO/\varpi^n)$. Then we have natural surjective ring homomorphisms
\[ \TT^{(\kappa)}_n \twoheadrightarrow  \TT^{(\kappa)}_\cO / I_n  \textnormal{ and } \TT^{(\kappa)}\twoheadrightarrow \TT^{(\kappa)}_\cO, \]
the latter coming from the fact that the intersection $\bigcap_{n} I_n$ is zero.
For sufficiently large~$\kappa$, both homomorphisms are isomorphisms due to Lemma~\ref{lem:lifting}.
However, this need no longer be true in our principal case of interest $\kappa=1$ since the inclusions 
\[  M_1(\fn ,\cO) \otimes_\cO (\cO/\varpi^n) \hookrightarrow M_1(\fn , \cO/\varpi^n) \textnormal{ and  }  
  M_1(\fn ,\cO) \otimes_\cO (K/\cO) \hookrightarrow M_1(\fn , K/\cO)  \]
need not be isomorphisms, in general. The kernel of $\TT^{(1)}\twoheadrightarrow \TT^{(1)}_\cO$ is a  finitely  generated  torsion $\cO$-module, which is isomorphic to the kernel of $\TT^{(1)}_n \to \TT^{(1)}_{\cO}/I_n$ for $n\in \NN$ sufficiently large. 
Recall that multiplication by the Hasse invariant $h_n$ allows us to see $M_1(\fn ,\cO/\varpi^n)$ inside $M_{\kappa_n}(\fn ,\cO/\varpi^n)$ equivariantly for all Hecke operators $T_\fq$ and $\langle\fq\rangle$ for $\fq \nmid \fn p$, yielding a surjection $\TT^{(\kappa_n)}_n \twoheadrightarrow \TT^{(1)}_n$  (see Proposition \ref{prop:doubling}).
For a prime $\fp \mid p$, consider also the  Hecke algebra: 
\begin{align}\label{eq:TT-tilde}
\widetilde{\TT}^{(\kappa_n)}_n=\TT^{(\kappa_n)}_n[T_\fp^{(\kappa_n)}] \subset  \End_\cO(M_{\kappa_n}(\fn ,\cO/\varpi^n))\big).
\end{align}

\begin{cor}\label{cor:doubling}
Let $\fp\mid p$. Then for any $n \in \NN$,
there is a surjection sending $T^{(\kappa_n)}_{\fp}$ to $U$ (considered as a polynomial variable): 
\[\widetilde{\TT}^{(\kappa_n)}_{n} \twoheadrightarrow \TT^{(1)}_{\cusp,n}[T_\fp^{(1)},U]/\big(U^2 - T_\fp^{(1)}U + \diam{\fp}\big).\]
 The same statement holds after replacing $\TT^{(1)}_{\cusp,n}$ by $\TT^{(1)}_n$, provided $(p-1)\nmid e_\fp$.
\end{cor}

\begin{proof}
The injection from Proposition~\ref{prop:doubling} gives a  morphism $ \widetilde{\TT}^{(\kappa_n)}_{n}\to \End_\cO\big(S_{1}(\fn ,\cO/\varpi^n)^{\oplus 2}\big)$
of $\cO$-algebras compatible with $T_\fq$ and $\langle\fq\rangle$  for all $\fq \nmid \fn p$ and, hence, we get a surjection $\TT^{(\kappa_n)}_{n} \twoheadrightarrow \TT^{(1)}_{\cusp,n}$.
Moreover,  $T_\fp ^{(\kappa_n)}$ acts on the image of $S_1(\fn ,\cO/\varpi^n)^{\oplus 2}$ via the matrix $\left(\begin{smallmatrix} T_\fp^{(1)} & 1  \\ -\diam{\fp} & 0 \end{smallmatrix}\right)$, whence
it is annihilated by its  characteristic polynomial   $U^2 - T_\fp^{(1)}U + \diam{\fp}$ and does not satisfy any non-trivial linear relation over  $\TT^{(1)}_{n}[T_\fp^{(1)}]$, thus proving the existence of the desired homomorphism. 
For the surjectivity, let us observe that the image of $S_1(\fn ,\cO/\varpi^n)^{\oplus 2}$ is contained in the 
$T_\fp^{(\kappa_n)}$-ordinary subspace of $M_{\kappa_n}(\fn ,\cO/\varpi^n)$, and that the  endomorphism
$T_\fp^{(\kappa_n)} + \diam{\fp}(T_\fp^{(\kappa_n)})^{-1}$ of the latter space acts on the former as $\left(\begin{smallmatrix} T_\fp^{(1)} & 0  \\ 0& T_\fp^{(1)} \end{smallmatrix}\right)$.
Finally, assuming $(p-1)\nmid e_\fp$ allows us to apply Proposition~\ref{prop:doubling} with $M_1(\fn ,\cO/\varpi^n)$ instead of $S_1(\fn ,\cO/\varpi^n)$, leading to the validity of the result with $\TT^{(1)}_{\cusp,n}$ replaced by $\TT^{(1)}_n$.
\end{proof}

\section{Pseudo-representations for weight $1$ Hecke algebras}\label{s:pseudo-CS}

\subsection{Pseudo-representations of degree $2$}

In this section, we recall some definitions due to Che\-ne\-vier~\cite{C} and  Calegari--Specter~\cite{CaSp}. 

\begin{defi}
\label{pseudodef}
Let $R$ be a complete Noetherian local $\cO$-algebra with maximal ideal~$\fm$ and residue field $R/\fm = \FF$ considered
with its natural $\fm$-adic topology.
An {\em $R$-valued pseudo-representation of degree~$2$} of $\rG_F$ is a tuple $P=(T,D)$ consisting of continuous maps
$T,D: \rG_F \to R$ such that
\begin{enumerate}[(i)]
\item\label{prop1} $D$ is a  group homomorphism $\rG_F \to R^\times$,
\item\label{prop2}  $T(1)=2$ and $T(gh) = T(hg) = T(g)T(h) - D(g) T(g^{-1}h)$ for all $g,h \in \rG_F$.
\end{enumerate}

We extend $T : \rG_F \to R$ to an $R$-linear map $R[\rG_F] \to R$ and we denote this map by $T$ as well. 

Given $g \in \rG_F$, we define $D(g-1) := D(g)-T(g)+1$.

The pseudo-representation $P=(T,D)$ is said to be {\em unramified at~$\fp$} if
$D(h-1) = T(g(h-1))=0$ for all $g \in \rG_F$ and all $h \in \rI_\fp$.
\end{defi}

 Any continuous representation $\rho:\rG_F \to \GL_2(R)$ yields a degree~$2$
pseudo-represen\-tation $P_\rho = (\tr \circ \rho, \det \circ \rho)$. The converse is true when the semi-simple  representation $\bar\rho:\rG_F \to \GL_2(\FF)$ corresponding to the residual pseudo-representation is absolutely irreducible (see  \cite[Thm. 2.22]{C}).  
Further,  if $\rho$ is unramified outside a finite set of places $S$, then so is $P_\rho$. Again, the converse is true  in the residually absolutely irreducible case. This can be seen by applying  {\it loc.~cit.}  to the Galois group of the maximal extension of $F$ unramified outside $S$ over $F$. 

We  introduce  a  notion of ordinarity inspired from  Calegari--Specter~\cite{CaSp}.

\begin{defi}\label{defi:ordinary}
Let $\widetilde{P}=(P,\alpha_\fp)$ with $P=(T,D)$ a degree~$2$ pseudo-repre\-sen\-tation of $\rG_F$ over $R$ and $\alpha_\fp\in R$
a root of $X^2 - T(\Frob_\fp )X+D(\Frob_\fp ) \in R[X]$.

We say that $\widetilde{P}$ is {\em ordinary at~$\fp$} of  weight $\kappa\geqslant 1$, if for all $h, h' \in \rI_\fp$ and all $g \in \rG_F$ we have
\begin{enumerate}[(i)]
\item $D(h-1)=0$ and $T(h-1)= \chi_p^{\kappa-1}(h) -1$, where $\chi_p$ denotes the $p$-adic cyclotomic character,   
\item $T\big(g(h-\chi_p^{\kappa-1}(h)) (h'\Frob_\fp  - \alpha_\fp)\big) = 0$. 
\end{enumerate}
\end{defi}

\begin{rem}
Note that our notion of $\fp$-ordinary pseudo-representations implies the one of Calegari--Specter (\cite[Def. 2.5]{CaSp}).
Let $P = (T, D) : \rG_F \to R^2$ be a degree $2$ pseudo-representation and let  $(\bar T, \bar D) : \rG_F \to \FF^2$ be its residual pseudo-representation. 
Suppose there exists a lift $\Frob_{\fp} \in \rG_{F_{\fp}}$ of the arithmetic Frobenius at $\fp$ such that the polynomial $X^2-\bar T(\Frob_{\fp}) X + \bar D(\Frob_{\fp})$ has \emph{distinct} roots in $\FF$.
Then $(P,\alpha_{\fp})$ is a $\fp$-ordinary pseudo-representation in the sense of Definition~\ref{defi:ordinary} if and only if it is $\fp$-ordinary in the sense of Calegari--Specter.
However, if this hypothesis does not hold, then we expect that the two notions are not equivalent.
\end{rem}

Let $\overline{P}=(\overline{T},\overline{D}): \rG_F \to\FF^2$ be a fixed degree $2$ pseudo-representation unramified outside $\fn p\infty$.
Denote by $P^\ps=(T^\ps,D^\ps): \rG_F \to (R^\ps)^2$ the universal deformation of~$\overline{P}$ unramified outside $\fn p\infty$ in the category of complete Noetherian local $\cO$-algebras with residue field $\FF$ and consider the quotient 
\begin{align}\label{eq:pseudo-ord-univ}
R^\ps[X]/(X^2 - T^\ps(\Frob_\fp )X+D^\ps(\Frob_\fp ))\twoheadrightarrow  R^\ord
\end{align}
which classifies pairs $(P,\alpha_\fp)$ such that $P$ is a deformation of~$\overline{P}$ unramified outside $\fn p \infty$ and $(P,\alpha_\fp)$ is ordinary at~$\fp$ of weight $\kappa$. 
The universal ring $R^\ord$, classifying deformations of~$\overline{P}$ which are unramified outside $\fn p \infty$ and are ordinary at~$\fp$ of weight~$\kappa$, is the quotient of the ring $R^\ps[X]/(X^2 - T^\ps(\Frob_\fp )X+D^\ps(\Frob_\fp ))$ by the ideal generated by the set
\[ \left\{D^\ps(h-1), T^\ps(h-1)- \chi_p^{\kappa-1}(h) +1, T^\ps\big(g(h-\chi_p^{\kappa-1}(h)) (h' \Frob_\fp  - X)\big) \;|\; h, h' \in \rI_\fp, g \in \rG_F\right\}\]
and a  direct computation shows that $R^\ord$ is independent of the choice of $\Frob_\fp$.

Note that $R^\ord$ is a finite $R^\ps$-algebra. 
As $R^\ps$ is a local ring, it follows that $R^\ord$ is a semi-local ring and all of its maximal ideals contain the unique maximal ideal $\fm^{\ps}$ of $R^{\ps}$.
After going modulo $\fm^{\ps}$ in $R^\ord$, it is easy to see, using the description of the ideal from the previous paragraph, that the number of maximal ideals of $R^\ord$ is the number of distinct $\alpha \in \FF$ such that $(\overline{P},\alpha)$ is a $\fp$-ordinary pseudo-representation of weight $\kappa$.

Now suppose $\overline{P}$ is unramified at $\fp$ and $\kappa \equiv 1 \pmod{p-1}$. Then  we have
\[\overline{T}\big(g(h-\chi_p^{\kappa-1}(h)) (h' \Frob_\fp  - X)\big) = \overline{T}\big(g(h-1) h' \Frob_\fp\big) - X\overline{T}(g(h-1)) = \overline{T}(h'\Frob_\fp g(h-1)) = 0.\]
Here we are repeatedly using the fact that $\overline{T}(g(h-1))=0$ for all $g \in G_F$ and $h \in \rI_\fp$, which is a consequence of the assumption that $\overline{P}$ is unramified at $\fp$.
Thus, in this case, we see that $(\overline{P},\alpha)$ is a $\fp$-ordinary pseudo-representation of weight $\kappa$ if and only if $\alpha$ is a root of the polynomial $X^2 - \overline{T}(\Frob_\fp )X+ \overline{D}(\Frob_\fp )$.
Hence,  in this case, $R^\ord$  is a semi-local Noetherian ring with two maximal ideals if the polynomial $X^2 - \overline{T}(\Frob_\fp )X+ \overline{D}(\Frob_\fp )$ has two distinct roots and it is a local Noetherian ring otherwise.

\subsection{Existence of an ordinary Hecke algebra-valued pseudo-representation}\label{sec:ord}
We continue to use the notation from \S\ref{sec:doubling}.
Let $\fm$ be any maximal ideal of $\TT^{(1)}$ (or equivalently of $\TT^{(1)}_n$ for some $n$) and 
denote also by $\fm$ the maximal ideals of $\TT^{(\kappa_n)}$ and $\TT^{(\kappa_n)}_n$ 
defined as the pull-back of $\fm \subset \TT^{(1)}_n$.

\begin{lem}\label{lem:exps}
There exists a  $\TT^{(\kappa_n)}_{n,\fm}$-valued pseudo-representation $P^{(\kappa_n)}_{n,\fm}$ of $\rG_F$ of degree $2$
which is unramified at all  primes $\fq \nmid \fn p $ and $P^{(\kappa_n)}_{n,\fm}(\Frob_\fq)=(T_\fq, \diam{\fq})$. 
In particular, after replacing  $\cO$ by a finite unramified extension, there exists a unique semi-simple Galois representation
\[ \rhobar_{\fm}: \rG_F \to \GL_2(\TT^{(1)}/\fm)\]
unramified outside $\fn p \infty$ satisfying
\[ \tr(\rhobar_{\fm}(\Frob_\fq)) = T_\fq \pmod{\fm} \textnormal{ and  }\det(\rhobar_{\fm}(\Frob_\fq)) = \diam{\fq} \pmod{\fm} \]
for all primes $\fq \nmid \fn p$. 
\end{lem}

\begin{proof}
After enlarging $K$, we may assume that it contains all the eigenvalues of $\TT^{(\kappa_n)}$ acting on  $M_{\kappa_n}(\fn,\cO)$. 
The $\cO$-algebra $ \TT^{(\kappa_n)}$ generated by the Hecke operators outside the level and $p$ is torsion-free and reduced, hence 
$ \TT^{(\kappa_n)}_{\fm} \otimes_{\cO}K = \prod_{g\in \mathcal{N}} K$ where $\mathcal{N}$ denotes the set of newforms  occurring in $M_{\kappa_n}(\fn,\cO)_{\fm}$. 
As is well known, one can attach to each such eigenform $g$ a $\rG_F$-pseudo-representation $P_g$ of degree~$2$ unramified outside $\fn p \infty$ such that 
$P_g(\Frob_\fq)=(a(\fq,g),\psi_g(\fq) \Norm(\fq)^{\kappa_n-1})$  for all $\fq \nmid \fn p$, 
where  $\langle \fq \rangle g= \psi_g(\fq) g$.  
Since  the natural homomorphism $ \TT^{(\kappa_n)}_{\fm} \to \TT^{(\kappa_n)}_{\fm} \otimes_{\cO}K$ is injective,  in view of the Chebotarev Density Theorem, we obtain a $\TT^{(\kappa_n)}_{\fm}$-valued $\rG_F$-pseudo-representation $P^{(\kappa_n)}_{\fm}$  unramified outside $\fn p \infty$ 
such that $P^{(\kappa_n)}_{\fm}(\Frob_\fq) = (T_{\fq}, \langle \fq \rangle\Norm(\fq)^{\kappa_n-1})$ for all $\fq \nmid \fn p$ (see  \cite[Cor.~1.14]{C}). 

Note that $\rN(\fq)^{\kappa_n-1} \equiv 1 \pmod{\varpi^n}$.
Composing $P^{(\kappa_n)}_{\fm}$ with the surjection $\TT^{(\kappa_n)}_{\fm} \twoheadrightarrow \TT^{(\kappa_n)}_{n,\fm}$, we get the desired pseudo-representation.
Finally  $\TT^{(\kappa_n)}_n/\fm=\TT^{(1)}_n/\fm$ along with \cite[Thm.~A]{C} finishes the proof  of the lemma.
\end{proof}

Let $R^\ps_\fm $ be the universal deformation ring of the corresponding degree $2$ pseudo-repre\-sen\-tation
$\overline{P}_{\fm} = (\tr\circ \rhobar_{\fm}, \det\circ \rhobar_{\fm})$ unramified outside  $\fn p\infty$ in the category of complete Noetherian local $\cO$-algebras with residue field $\FF$
(chosen large enough in order to contain the residue field of $\TT^{(1)}_\fm$). 
Using the surjection  $\TT^{(\kappa_n)}_{\fm} \twoheadrightarrow \TT^{(\kappa_n)}_{n,\fm} \twoheadrightarrow \TT^{(1)}_{n,\fm}$
and then passing to the projective limit $\TT^{(1)}_{\fm}=\varprojlim\limits_n \TT^{(1)}_{n,\fm}$, we obtain  the 
following result.

\begin{cor}\label{cor:ex}
For any  maximal ideal~$\fm$ of $\TT^{(1)}$, there exists a $\TT^{(1)}_{\fm}$-valued pseudo-representation $P^{(1)}_{\fm}$ of $\rG_F$  of  degree $2$  which is  unramified for all  primes $\fq \nmid \fn p$ and  $P^{(1)}_{\fm}(\Frob_\fq)=(T_\fq, \diam{\fq})$. It yields a surjection $R^\ps_\fm  \twoheadrightarrow  \TT^{(1)}_{\fm}$.
\end{cor}

Note that for a maximal ideal $\fm$ of  $\TT^{(\kappa_n)}_n$, the algebra $\widetilde{\TT}^{(\kappa_n)}_{n,\fm}$ 
is in general only semi-local (see \eqref{eq:TT-tilde}). By the main result of~\cite{DiWi}, $\rhobar_{\fm}$ is unramified at~$\fp$, allowing us to consider the ideal 
\begin{align}\label{eq:m-tilde}
\widetilde{\fm}=
\left(\fm,  (T_\fp^{(\kappa_n)})^2 - \widehat{\tr(\rhobar_{\fm}(\Frob_\fp ))} T_\fp^{(\kappa_n)} +\widehat{ \det(\rhobar_{\fm}(\Frob_\fp ))}\right)
\subset \widetilde{\TT}^{(\kappa_n)}_n,
\end{align}
where $\widehat{\tr(\rhobar_{\fm}(\Frob_\fp ))}$ and $\widehat{ \det(\rhobar_{\fm}(\Frob_\fp ))}$ are some lifts of $\tr(\rhobar_{\fm}(\Frob_\fp ))$ and  $\det(\rhobar_{\fm}(\Frob_\fp ))$, respectively in $\TT^{(\kappa_n)}_n$. 
Note that the ideal $\widetilde{\fm}$ does not depend on the choices of these lifts.

Let $\widetilde{\TT}^{(\kappa_n)}_{n,\widetilde{\fm}}$ be the completion of $\widetilde{\TT}^{(\kappa_n)}_{n}$ with respect to $\widetilde{\fm}$.
The algebra $\widetilde{\TT}^{(\kappa_n)}_{n,\widetilde{\fm}}$ then has at most two local components. 
Let $R^\ord_\fm $ be the universal $\cO$-algebra classifying deformations  of~$\overline{P}_{\fm}$ 
 unramified outside primes dividing $\fn p\infty$ and ordinary  at~$\fp$ of weight $1$
  (see \eqref{eq:pseudo-ord-univ}). 

\begin{lem}\label{lem:ordps}
There exists a $\fp$-ordinary $\widetilde{\TT}^{(\kappa_n)}_{n,\widetilde{\fm}}$-valued pseudo-representation 
$\widetilde{P}_{n,\widetilde\fm}^{(\kappa_n)}=(P_{n,\fm}^{(\kappa_n)},T_\fp^{(\kappa_n)})$ of degree $2$ and weight $1$ of $\rG_F$ such that $P_{n,\fm}^{(\kappa_n)}(\Frob_\fq)=(T_\fq, \diam{\fq})$ for all $\fq \nmid \fn p$. It yields a surjection $R^\ord_\fm \twoheadrightarrow \widetilde{\TT}^{(\kappa_n)}_{n,\widetilde{\fm}}$.
\end{lem}

\begin{proof}
Let $\widetilde{\TT}^{(\kappa_n)} = \TT^{(\kappa_n)}[T^{(\kappa_n)}_\fp]$ and denote also by $\widetilde\fm$ the ideal of $\widetilde\TT^{(\kappa_n)}$ defined as the pull-back of $\widetilde\fm \subset\widetilde \TT^{(\kappa_n)}_n$. Let $\widetilde{\TT}^{(\kappa_n)}_{\widetilde{\fm}}$ be the completion of $\widetilde{\TT}^{(\kappa_n)}$ with respect to $\widetilde{\fm}$.

We have  $\widetilde{\TT}^{(\kappa_n)}_{\widetilde{\fm}} \otimes_{\cO}K = \prod_{g\in \widetilde{\mathcal{N}}} K$, where $\widetilde{\mathcal{N}}$  denotes the subset of $\mathcal{N}$ (see  the proof of Lemma \ref{lem:exps}) consisting of newforms  occurring in $M_{\kappa_n}(\fn,\cO)_{\widetilde\fm}$. As $\fp$ does not divide $\fn$, any $g\in \widetilde{\mathcal{N}}$ is an eigenvector for
$T_\fp^{(\kappa_n)}$ (resp. $\langle \fp \rangle$) whose  eigenvalue $a(\fp,g)$ (resp. $\psi_g(\fp)$) is necessarily a $p$-adic unit by \eqref{eq:m-tilde}, {\it i.e.},   $g$ is $\fp$-ordinary. By a result due to  Hida and Wiles, when $g$ is ordinary at all places dividing $p$, and to 
Saito \cite{saito} and Skinner \cite{skinner} in general, $p$-adic Galois representation $\rho_g$ attached to $g$ is ordinary at~$\fp$, {\it i.e.}, its restriction to $\rG_{F_\fp}$ has a one-dimensional unramified quotient on which $\Frob_\fp$ acts by the (unique)
$p$-adic unit root $\alpha_{\fp, g}$ of the Hecke polynomial $X^2-a(\fp,g) X + \psi_g(\fp) \Norm(\fp)^{\kappa_n-1}$. 
This implies that  $\alpha_{\fp, g}$ is also a root  of 
$X^2-\tr(\rho_g)(\Frob_\fp)X+\det(\rho_g)(\Frob_\fp)$, for {\it any} choice of a Frobenius element $\Frob_\fp\in \rG_{F_\fp}$. 
Thus, the pseudo-representation $P_g=(\tr(\rho_g), \det(\rho_g))$ 
is  $\fp$-ordinary of weight $\kappa_n$ with respect to  $\alpha_{\fp, g}$ in the sense of Definition~\ref{defi:ordinary}. 

Since  $\widetilde{\TT}^{(\kappa_n)}_{\widetilde{\fm}}$ is a semi-local finite $\cO$-algebra, applying Hensel's lemma to each local component shows that the polynomial $X^2 - T_{\fp}^{(\kappa_n)} X +  \langle \fp \rangle\Norm(\fp)^{\kappa_n-1}$ admits a unique 
 unit root $U$  in $\widetilde{\TT}^{(\kappa_n)}_{\widetilde{\fm}}$.
By the Chebotarev Density Theorem, gluing the $\fp$-ordinary pseudo-representations $\widetilde{P}_g=(P_g, \alpha_{\fp, g})$ for all  $g\in \widetilde{\mathcal{N}}$ 
gives us a $\widetilde{\TT}^{(\kappa_n)}_{\widetilde{\fm}}$-valued $\fp$-ordinary pseudo-representation $(P_{\fm}^{(\kappa_n)},U)$ of weight $\kappa_n$ such that $P_{\fm}^{(\kappa_n)}(\Frob_\fq)=(T_\fq, \diam{\fq}\Norm(\fq)^{\kappa_n-1})$ for all $\fq \nmid \fn p$.
We have $\chi_p^{\kappa_n-1} (g) \equiv 1 \pmod{\varpi^n}$ for all $g \in G_F$.
Hence, the reduction of $(P_{\fm}^{(\kappa_n)},U)$ to $\widetilde{\TT}^{(\kappa_n)}_{n,\widetilde{\fm}}$ is a $\fp$-ordinary pseudo-representation of weight $1$.
Note that by Hensel's Lemma,  $U$ reduces  to $T_\fp^{(\kappa_n)}$ in $\widetilde{\TT}^{(\kappa_n)}_{n,\widetilde{\fm}}$,
since the former (resp.\ the latter) is the unique unit root of $X^2 - T_{\fp}^{(\kappa_n)} X +  \langle \fp \rangle\Norm(\fp)^{\kappa_n-1}$
in $\widetilde{\TT}^{(\kappa_n)}_{\widetilde{\fm}}$ (resp. in $\widetilde{\TT}^{(\kappa_n)}_{n,\widetilde{\fm}}$). 
As $\Norm(\fq)^{\kappa_n-1} \equiv 1 \pmod{\varpi^n}$ for all $\fq \nmid \fn p$, this completes the proof of the lemma.
\end{proof}

\subsection{Proof of the main theorem}
In the proof of Theorem~\ref{thm:main} we can assume without loss of generality that $\fn$  satisfies \eqref{ass:21}, because given any prime $q$, the Hecke algebra in level~$\fn$ is a quotient of the one in level~$\fn q$.
Moreover, since the algebra $\TT^{(1)}$ is semi-local, equal to the product of $\TT^{(1)}_{\fm}$ where $\fm$ runs over its maximal ideals, 
it is enough to prove the theorem after localisation at~$\fm$.

Recall that in Corollary~\ref{cor:ex} we constructed a $\TT^{(1)}_\fm$-valued pseudo-representation $P^{(1)}_\fm=(T,D)$ of $\rG_F$, whose image under  the surjective homomorphism $\TT^{(1)}_\fm \twoheadrightarrow \TT^{(1)}_{\cusp,\fm}\twoheadrightarrow \TT^{(1)}_{\cusp,n,\fm}$
will be denoted by  $P^{(1)}_{n,\fm}=(T_n,D_n)$, for $n\in \NN$.   This gives the first row of the following commutative diagram:
\begin{align}\label{eq:diagpf}
 \xymatrix@R=1cm@C=2cm{
R^\ps_\fm  \ar@{->>}[r]^{\textnormal{Cor.~\ref{cor:ex}}} \ar@{->}[d]& \TT^{(1)}_{\cusp,\fm} \ar@{->>}[r] & \TT^{(1)}_{\cusp,n,\fm} \ar@{^(->}[d] \\
R^\ord_\fm \ar@{->>}[r]^{\textnormal{Lemma~\ref{lem:ordps}}} & \widetilde{\TT}^{(\kappa_n)}_{n,\widetilde{\fm}} \ar@{->>}[r]^(.3){\textnormal{Cor.~\ref{cor:doubling}}} & \TT^{(1)}_{\cusp,n,\fm}[T_\fp^{(1)},U]/(U^2 - T_\fp^{(1)} U+\diam{\fp}). 
}\end{align}
The morphisms in the second row come from Lemma~\ref{lem:ordps} and Corollary~\ref{cor:doubling}.
Combining them, we see that $\widetilde{P}^{(1)}_{n,\fm} = (P^{(1)}_{n,\fm}, U)$ is a $\fp$-ordinary pseudo-representation of weight $1$.

We now perform the key `doubling' step, as presented in \cite[Prop. 2.10]{CaSp}, and slightly improved upon, since the
surjectivity of the composed map $R^\ord_\fm\to  \TT^{(1)}_{\cusp,n,\fm}[T_\fp^{(1)},U]/(U^2 - T_\fp^{(1)} U+\diam{\fp})$ will not be used in the sequel. 
One has 
\[\TT^{(1)}_{\cusp,n,\fm}[T_\fp^{(1)},U]/(U^2 - T_\fp^{(1)} U+\diam{\fp})=\TT^{(1)}_{\cusp,n,\fm}[T_\fp^{(1)}]\oplus U\cdot \TT^{(1)}_{\cusp,n,\fm}[T_\fp^{(1)}]. \]
Since $\widetilde {P}^{(1)}_{n,\fm}$ is ordinary at $\fp$ of weight $1$, for all $g\in \rG_F$ and $h\in \rI_\fp$ the following equality holds: 
\[T_n(gh \Frob_\fp)- T_n(g \Frob_\fp)= U(T_n(gh)-T_n(g))\in \TT^{(1)}_{\cusp,n,\fm}[T_\fp^{(1)}]\cap U\TT^{(1)}_{\cusp,n,\fm}[T_\fp^{(1)}]=\{ 0\}, \]
hence $T_n(gh)=T_n(g)$, {\it i.e.}, $P^{(1)}_{n,\fm}$ is unramified at~$\fp$. 

Note that $U$ satisfies the following relations
\[ U^2 - T_\fp^{(1)} U + \langle \fp \rangle = 0 \text{ and }  U^2 - T_n(\Frob_\fp) U + D_n(\Frob_\fp) = 0 \]
 in the ring $\TT^{(1)}_{\cusp,n,\fm}[T_\fp^{(1)},U]/(U^2 - T_\fp^{(1)} U+\diam{\fp})$. 
Indeed, the second relation follows from the fact that $(P^{(1)}_{n,\fm},U)$ is a $\fp$-ordinary pseudo-represen\-ta\-tion of weight $1$.
As the former polynomial is minimal, one obtains the  desired equality $(T_n(\Frob_\fp), D_n(\Frob_\fp))=(T_\fp^{(1)},\langle \fp \rangle)$, 
in particular $T_\fp^{(1)}\in \TT^{(1)}_{\cusp, n,\fm}$.  Letting $n$ vary  finishes the proof of Theorem~\ref{thm:main} for $\TT^{(1)}_{\cusp,\fm}$. 

In order to obtain the theorem for $\TT^{(1)}_{\fm}$, we replace $\TT^{(1)}_{\cusp,\fm}$ by $\TT^{(1)}_{\fm}$, $\TT^{(1)}_{\cusp,n,\fm}$ by $\TT^{(1)}_{n,\fm}$, and $S_1(\fn ,\cO/\varpi^n)$ by $M_1(\fn ,\cO/\varpi^n)$ throughout. The arguments continue to work if we assume that $p-1$ does not divide~$e_\fp$, which is used in Corollary~\ref{cor:doubling}.

\begin{cor}\label{cor:onefull}
Let $\fp\mid p$. Then $T_\fp^{(1)}  \in \TT_\cusp^{(1)}$, {\it i.e.},  for all $n \in \NN$, the Hecke operator
$T_\fp^{(1)}$ acts on $S_1(\fn ,\cO/\varpi^n)$ by an element of $\TT^{(1)}_{\cusp,n}$.
Moreover, if $(p-1)\nmid e_\fp$, then one also has $T_\fp^{(1)}  \in \TT^{(1)}$.
\end{cor}

\subsection{Non-Eisenstein ideals}

\begin{defi}\label{def:eis}
A maximal ideal   $\fm$  of $\TT^{(\kappa)}$ (or of $\TT_n^{(\kappa)}$) is called Eisenstein if the corresponding $(\TT_{\fm}^{(\kappa)}/\fm)$-valued  pseudo-representatation  of $\rG_F$ is the sum of two $(\TT_{\fm}^{(\kappa)}/\fm)$-valued characters. 
\end{defi}

We now prove that in  the non-Eisenstein case it suffices to consider  the cuspidal Hecke algebra. 

\begin{prop}\label{prop:loc-eis}
The localisation of the natural surjection $\TT_n^{(\kappa)}\twoheadrightarrow \TT^{(\kappa)}_{n,\mathrm{cusp}}$ at any non-Eisenstein maximal ideal $\fm$ of $\TT_n^{(\kappa)}$ is an isomorphism. 
\end{prop}

\begin{proof}
It suffices to prove that the localisation of $M_\kappa(\fn,\cO/\varpi^n)/S_\kappa(\fn,\cO/\varpi^n)$ at a non-Eisenstein ideal vanishes.
By multiplication by a suitable power of $h_n$ which does not vanish at any cusp  (see Lemma \ref{lem:Hasse}),
 we can assume that $\kappa$ is sufficiently large so that Lemma~\ref{lem:lifting} applies yielding 
  $M_\kappa(\fn,\cO/\varpi^n)=M_\kappa (\fn,\cO) \otimes_\cO (\cO/\varpi^n)$. Hence the natural Hecke equivariant morphism 
\[M_\kappa(\fn,\cO)/S_\kappa(\fn,\cO) \to M_\kappa(\fn,\cO/\varpi^n)/S_\kappa(\fn,\cO/\varpi^n)\]
is surjective. The former, however, can be Hecke equivariantly embedded into $M_\kappa(\fn,\CC)/S_\kappa(\fn,\CC)$ which is 
well known to be generated by Eisenstein series whose Galois representations are reducible. This proves the proposition.
\end{proof}

Henceforth we assume $\fm$ to be a non-Eisenstein ideal of $\TT^{(1)}$, so that the corresponding 
residual Galois representation $\rhobar_\fm$ is absolutely irreducible. 
Therefore, by combining Theorem~\ref{thm:main} with a result of Chenevier (\cite[Thm. 2.22]{C}), we get a  representation 
\[\rho_\fm : \rG_F \to \GL_2(\TT^{(1)}_\fm),\]
unramified outside $ \fn p \infty $, and  uniquely characterised  by the property that for all primes $\fq \nmid \fn p$ one has  $\tr(\rho_\fm(\Frob_\fq)) = T_\fq$ and $\det(\rho_\fm(\Frob_\fq))=\langle \fq \rangle$.   
By combining Theorem~\ref{thm:main} with Proposition~\ref{prop:loc-eis}, we deduce that the pseudo-representation $P^{(1)}_{\fm}$
is unramified at all primes $\fp \mid p$ and,  by the discussion after Definition~\ref{pseudodef}, we conclude that 
$\rho_{\fm}$ is unramified at these primes as well. 
Let $S$ be the set of places of $F$ dividing $\fn\infty$ and let $R^S_{F,\rhobar_\fm}$ be the universal deformation ring of $\rhobar_\fm$ unramified outside $S$ in the category of complete Noetherian local $\cO$-algebras with residue field $\FF$. 
Hence $\rho_{\fm}$ induces an $\cO$-algebra homomorphism $R^S_{F,\rhobar_\fm} \rightarrow \TT^{(1)}_\fm$.

As $\TT^{(1)}$ is generated by $T_\fq$ and $\langle \fq \rangle$ for $\fq \nmid \fn p$ as an $\cO$-algebra, we obtain the following result. 

\begin{cor}\label{cor:RT}
There exists a surjective homomorphism $R^S_{F,\rhobar_\fm} \twoheadrightarrow \TT^{(1)}_\fm$ of $\cO$-algebras.
\end{cor}

\bibliography{References}
\bibliographystyle{siam}

\end{document}